\documentclass[a4paper,12pt]{article}
\pdfoutput=1
\usepackage{amsfonts}
\usepackage{amssymb}
\usepackage{amsthm}
\usepackage{graphicx} 

\newtheorem{theorem}{Theorem}
\newtheorem{lemma}[theorem]{Lemma}

\begin{document}
\begin{center}{\Large\bf The simple graph threshold number $\sigma(r,s,a,t)$} \end{center}
\bigskip
\centerline{A.J.W. Hilton and A. Rajkumar}
\bigskip

\begin{abstract}
For $d \ge 1$, $s \ge 0$ a $(d, d+s)$-{\em graph} is a graph whose degrees all lie in the interval $\{d, d+1, \ldots, d + s\}$. For $r \ge 1$, $a \ge 0$, an $(r, r+a)$-{\em factor} of a graph $G$ is a spanning $(r, r+a)$-subgraph of $G$.  An $(r, r+a)$-{\em factorization} of a graph $G$ is a decomposition of $G$ into edge-disjoint $(r, r+a)$-factors. A graph is $(r, r+a)$-{\em factorable} if it has an $(r, r+a)$-factorization.

Let $\sigma(r, s, a, t)$ be the least integer such that, if $d \ge \sigma(r, s, a, t)$, then every $(d, d+s)$-simple graph $G$ is $(r,r+a)$-factorable with $x$ factors for at least $t$ different values of $x$.

In this paper we evaluate $\sigma(r,s,a,t)$ for all values of $r, s, a$ and $t$. We also show that if $a \ge 2$ and $r \ge 1$, then, when $r$ is even and $a$ is odd, every $(d, d+s)$-simple graph $G$ has an $(r, r+a)$-factorization with $x$ factors if and only if 
$$ \frac{d+s}{r+a}\, < x \le \frac{d}{r}\,,$$ 
and we prove similar statements for other parities of $r$ and $a$.
\end{abstract}

\section{Introduction}\label{sec:introduction}

For $d\ge1$, $s\ge0$, a $(d, d+s)$-graph is a graph whose degrees all lie in the interval $\{d, d+1, \ldots, d+s\}$.  For $r \ge 1, a \ge 0$, an $(r,r+a)$-{\em factor} of a graph $G$ is a spanning $(r, r+a)$-subgraph of $G$.  An $(r, r+a)$-{\em factorization} of a graph $G$ is a decomposition of $G$ into edge-disjoint $(r, r+a)$-factors.  If $G$ has an $(r, r+a)$-factorization then we say it is $(r, r+a)$-{\em factorable}. Sometimes when there can be no confusion we refer simply to factors, rather than $(r, r+a)$-factors.

For $r \ge 0$, $s \ge 0$, $a \ge 0$ and $t \ge 1$, let $\sigma(r, s, a, t)$ be the least integer such that, if $d \ge \sigma(r, s, a, t)$, then every $(d, d+s)$-simple graph $G$ has an $(r, r+a)$-factorization into $x$ $(r, r+a)$-factors for at least $t$ different values of $x$.  The number $\sigma(r, s, a, t)$ is called the simple graph $(r, s, a, t)$-{\em threshold number}.  In this paper we evaluate $\sigma(r, s, a, t)$.

Let us illustrate our terminology with a few examples.  By Vizing's theorem \cite{Vizing64}, every simple $(d, d+s)$-graph has a $(0,1)$-factorization into $d + s + 1$ $(0,1)$-factors.  Thus
$$\sigma(0,s,1,1)=0 \; \; {\rm for} \;\;s \ge 0\,.$$
[In fact, you can deduce this without using Vizing's theorem.]  Similarly by Gupta's theorem \cite{Gupta66}, \cite{Hilton75}, for $d \ge 1$ every $d$-regular graph has a $(1, 2)$-factorization into $d-1$ $(1,2)$-factors.  Thus
$$ \sigma(1, 0, 1, 1) = 1\,.$$

To give an example illustrating the parameter $t$, we may take a proper edge-colouring with 30 colours of a $29$-regular simple graph $G$, which exists by Vizing's theorem.  Let us combine the colours in threes, so that there are 10 sets of combined colours.  This gives a $(2,3)$-factorization of $G$ with $10$ $(2,3)$-factors.  Now take the same $29$-regular graph and form an edge-covering with $28$ colours, so that each colour appears on an edge at each vertex; this exists by Gupta's theorem.  Combine these colours together in two's, so that there are $14$ sets of combined colours.  This gives another $(2,3)$-factorization of $G$, but this time there are $14$ $(2,3)$-factors.  It was shown in \cite{Hilton09} that in fact $G$ has a $(2,3)$-factorization with $x$ colours for each $x \in \{10, 11, 12, 13, 14\}$ and for  no other values of $x$.  Moreover, it was shown in \cite{Hilton09} that
$$\sigma(2, 0, 1, 5) = 28\,.$$
Thus there is a regular simple graph of degree $27$ which does not have a $(2,3)$-factorization with $x$ $(2,3)$-factors for $5$ different values of $x$, but, if $d \ge 28$, then every $d$-regular simple graph does have a  $(2,3)$-factorization with $x$ factors for at least $5$ different values of $x$.

\subsection{Analogous threshold numbers}\label{subsec:analogous}

The threshold numbers for several analogous concepts have already been evaluated.  Let $\beta(r, s, a, t)$ be the analogous threshold number for bipartite multigraphs, and let $\beta_s (r, s, a, t)$ be the analogous threshold number for bipartite simple graphs.  Let $\pi(r, s, a, t)$ be the analogous threshold number for {\em pseudographs} (also known  as general graphs, that is graphs where multiple edges and multiple loops are allowed).  Finally let $\mu(r, s, a, t)$ be the analogous threshold number for multigraphs (that is, pseudographs with no loops).

For $r, t \ge 1$ and $s, a \ge 0$ we define a number $N(r, a, s, t)$ by
$$ N(r, a, s, t) = r \left\lceil \frac{r t+s-1}{a} \right\rceil + (t - 1)r\,.$$

For bipartite graphs we showed \cite{Hilton10}, \cite{HiltonRaj1},

\begin{theorem} For $r, t \ge 1$ and $s, a \ge 0$,
$$ \beta(r, s, a, t) = \beta_s(r, s, a, t) = N(r, s, a, t)\,.$$
\label{T1}\end{theorem}
The quite easy arguments used to derive our results for bipartite graphs serve as a  template for our arguments for simple graphs. Our results on pseudographs were mainly a rather complicated deduction from the bipartite graph results.  Thus our results on pseudographs and our results for simple graphs are more or less independent, the only connection being via our results/arguments for bipartite graphs.

An easy deduction from Theorem \ref{T1} and first principles tells us:

\begin{lemma} For $r,t \ge 1$ and $s, a\ge 0$
$$N(r,a,s,t) = \beta_s(r,s,a,t) \le \sigma(r,a,s,t) \le \mu(r,s,a,t) \le \pi(r,s,a,t)\,.$$
\label{L2}\end{lemma}

\begin{proof} Each bipartite simple graph is {\em a fortiori} a simple graph, so 
$$ \beta_s(r,s,a,t) \; \le \; \sigma(r,s,a,t)\,.$$
The other inequalities follow similarly.
\end{proof}

For $r$ and $a$ both even we have \cite{Hilton10}, \cite{HiltonRaj1},

\begin{theorem} For $r, a \ge 2$, even, and $s \ge 0$, $t \ge 1$, 
$$ \sigma(r, s, a, t) = \mu(r, s, a, t) = \pi(r, s, a, t) = N(r, s, a, t)\,.$$
\label{T3}\end{theorem}
For pseudographs we have the following two theorems \cite{Hilton10}, \cite{HiltonRaj1}.  The first deals with the special cases when $a = 0$ or $1$.  First let us remark that the notation $\pi(r, s, 0, t) = \infty$ means that there is no smallest value of $d$, say $d = d_0$, such that, if $d \ge d_0$ then each $(d, d+s)$-pseudograph has an $(r, r+0)$-factorization with $x$ factors for at least $t$ values of $x$.

\begin{theorem} Let $r, s$ and $t$ be integers with $r$ and $t$ positive and $s$ non-negative.  Then 
$$\pi(r, s, 0, t) = \infty$$
and
$$\pi(r, s, 1, t) = \left\{ \begin{array}{cl} 2 & {\rm if} \;\; r = 2, s = 0 \;\;{\rm and} \;\; t = 1\,, \\
1 & {\rm if} \;\; r = 1, s = 0 \;\; {\rm and}\;\; t = 1\,,  \\
\infty & {\rm otherwise}. \end{array} \right.$$
\label{T4}\end{theorem}

There is a further special case when $a=2$ and $r$ is odd, which we hope to prove in a sequel to \cite{HiltonRaj1}.

\begin{theorem} Let $r, s$ and $t$ be integers with $r \ge 1$ odd, $t \ge 1$ and $s \ge 0$. Then 
$$\pi(r,s,2,t) \left\{ \begin{array}{cl} \infty & {\rm if} \;\; r \ge 1 \;\; {\rm and \; either} \;\; s > 1 \; {\rm or} \; t > 1\,, \\
1 & {\rm if} \;\;r=1\,, \; s\in \{0,1\} \; \; {\rm and} \;\; t=1\,. \end{array} \right.$$
\label{T5}\end{theorem}

For the cases when $a \ge 2$ not covered by Theorem \ref{T4} we have:

\begin{theorem} Let $r, s, a$ and $t$ be integers with $r$ and $t$ positive, $a \ge 2$ and $s$ non-negative.

\begin{itemize}
\item[{\rm 1.}] If $r$ and $a$ are both even, then
$$ \pi(r, s, a, t) = N(r, s, a, t)\,.$$
\item[{\rm 2.}] If $r$ and $a$ are both odd, then
$$\pi(r, s, a, t) = \left\{\begin{array}{l} N(r+1, s, a-1, t) - 1  \\ \hspace*{3.5cm} {\rm if}\;\ (r+1)t+s \not\equiv 2\;({\rm mod}\,a-1)\,,  \vspace{0.25cm}\\
N(r+1, s, a-1, t) - (r + 1) - 1 \\ \hspace*{3.5cm} {\rm if} \;\;(r+1) t+s \equiv 2 \; ({\rm mod}\,a-1)\,. \end{array} \right.$$

\item[{\rm 3.}] If $r$ is odd and $a$ is even, then
$$\pi(r, s, a, t) = \left\{\begin{array}{l} N(r+1, s, a-2, t) - 1 \\ \hspace*{3.5cm} {\rm if}\;\ (r+1)t+s \not\equiv 2, 3\;({\rm mod}\,a-2)\,,  \vspace{0.25cm}\\
N(r+1, s, a-2, t) - (r + 1) - 1 \\ \hspace*{3.5cm} {\rm if} \;\;(r+1) t+s \equiv 2, 3 \; ({\rm mod}\,a-2)\,. \end{array} \right.$$

\item[{\rm 4.}] If $r$ is even and $a$ is odd, then
$$\pi(r, s, a, t) = \left\{\begin{array}{ll} N(r,s,a-1,t) & {\rm if}\;\ rt+s \not\equiv 2\;({\rm mod}\,a-1)\,, \\
N(r, s, a-1, t) - r & {\rm if} \;\;rt+s \equiv 2 \; ({\rm mod}\,a-1)\,. \end{array} \right.$$

\end{itemize}

\label{T6}\end{theorem}

Cases (1), (2) and (4) are proved in \cite{HiltonRaj1}. The proof and result in Case (3) in \cite{HiltonRaj1} was wrong, since the special result when $(r+1) t+s \not\equiv 3$ $({\rm mod}\, a-2)$ was not noticed. The argument in this overlooked case is quite complicated, and it is hoped to publish it elsewhere as a sequel to \cite{HiltonRaj1}.

In the main the values of $\mu(r, s, a, t)$ are not known, and it certainly looks at present that they will be harder to determine than the results for simple graphs or pseudographs.  However, we do have some results.

In Theorem \ref{T3} we gave the evaluation when $r$ and $a$ are  both even; specifically:

\noindent ``If $r$ and $a$ are both even, $s \ge 0$, $t \ge 1$,  then 
$$\mu(r, s, a,t) = N(r, s, a, t)\,."$$

In the case when $s = 0$ and $t = 1$ we have \cite{FerencakHilton10}, \cite{FerencakHilton11}:

\begin{theorem} Let $r \ge 1$.  Then
$$\mu(r, 0, 1, 1) = r^2 + 1 \quad {\rm if} \;\; r \; {\rm is \;odd},$$
$$ \frac{3}{2} r^2 - 2r - 1 \le \mu(r, 0, 1, 1) \le \frac{3}{2}r^2 + 3r +1 \quad {\rm if} \;\; r  \;{\rm is \;even}.$$
\label{T7}\end{theorem}

\subsection{Known results for simple graphs}\label{subsec:known}
The first noteworthy result was due to Era in 1984 \cite{Era84} and Egawa in 1986 \cite{Egawa86}.

\begin{theorem} For integers $r \ge 3$,
$$\sigma(r, 0, 1, 1) = \left\{\begin{array}{cl} r^2 & {\rm if} \;\; r \;\; {\rm is \;even}, \\ 
r^2 + 1 & {\rm if} \;\; r\;\;  {\rm is\; odd}.\end{array}\right.$$

\label{T8}\end{theorem}

In 2009, extending this and other work by Hilton and Wojciechowski \cite{HiltonWojc05}, Hilton \cite{Hilton09} evaluated $\sigma(r, s, a, t)$ in the special case when $\alpha = 1$.

\begin{theorem} Let $r, s,$ and $t$ be integers with $r$ and $t$ positive and $s$ non-negative.  Then
$$\sigma(r, s, 1, t) = \left\{\begin{array}{ll} tr^2 + tr + sr - r & {\rm if} \;\; r \;\; {\rm is \; even \; and \;} s \in \{0,1\}, \\
tr^2 + tr + sr - r + 1 & {\rm if} \;\; r \;\; {\rm is \;odd \; and \;} s \in \{0, 1\}, \\
tr^2 + tr + sr + 1 & {\rm if} \;\; s \ge 2\,. \end{array}\right.$$
\label{T9}
\end{theorem}

For $r$ and $a$ even we have the following special case of Theorem \ref{T3}.

\begin{theorem} For $r, a \ge 2$, even, and $s \ge 0$, $t \ge 1$, 
$$ \sigma(r, s, a, t) = N(r, s, a,t)\,.$$
\label{T10}
\end{theorem}

\noindent In other words
$$\sigma(r, s,a, t) = r \left\lceil \frac{rt+s-1}{a}\right\rceil + (t-1)r\,.$$

This result ``dropped out'' of the same result for pseudographs.  The general method we use for simple graphs is quite different.

\subsection{New results for simple graphs}\label{subsec:new}

Theorem 9 gives the evaluation when $a = 1$. For general positive integer values of $a$ we have the following result.

\begin{theorem} Let $ r \ge 1$, $s \ge 0$, $a \ge 1$ and $t \ge 1$ be integers.  Then 

\begin{itemize}
\item[{\rm (i)}] If $r$ is odd and $a$ is even, then
$$\sigma(r,s,a,t) = \left\{ \begin{array}{ll}r\left\lceil \displaystyle \frac{tr+s+1}{a}\right\rceil + (t-1) r +1 & {\rm if} \;\;\; t \ge 2 \;\;\; {\rm or\;\; if} \\
&  t=1 \; {\rm and} \; a < r+s+1\;, \vspace{0.5cm} \\
r & {\rm if} \; t=1 \; {\rm and} \;a \ge r+s+1\,.\end{array} \right.$$ \\

\item[{\rm(ii)}] If $r$ is even and $a$ is even , then
$$\sigma(r,s,a,t) = r \left\lceil \frac{tr+s-1}{a} \right\rceil + (t-1)r\,.$$

\item[{\rm(iii)}] If $r$ is even and $a$ is odd, then
$$\sigma(r,s,a,t) = r \left\lceil \frac{tr+s}{a}\right\rceil + (t-1)r\,.$$

\item[{\rm(iv)}] If $r$ and $a$ are both odd, then
$$\sigma(r,s,a,t) = \left\{\begin{array}{ll} r \left\lceil \displaystyle \frac{tr+s}{a} \right\rceil = (t-1) r+1 & {\rm if} \; t \ge 2 \;\; {\rm or \;\; if} \\
& t=1 \; \; {\rm and} \;\; a < r+s \,, \vspace{0.5cm} \\
r & {\rm if} \;\; t=1 \;\; {\rm and} \;\; a \ge r+s\,. \end{array} \right.$$

\end{itemize}
\label{T11}\end{theorem}

The evaluation of (i) in Theorem \ref{T11} was attempted in \cite{HiltonRaj2}, but we thank C.J.H. McDiarmid for pointing out that that evaluation was wrong.

It will be noticed that in the case of simple graphs, the evaluations are all quite close to each  other (not far from $r \left\lceil \frac{tr+s}{a} \right\rceil + (t-1)r$ ) unlike the case of pseudographs where the denominator varies from $a-2$ to $a$ in the various cases.

Another unexpected point of interest for simple graphs is given in Theorem \ref{T12}.

\begin{theorem} Let $a \ge 1$, $r \ge 1$, $s \ge 0$.  Every $(d, d+s)$-simple graph $G$ has an $(r, r+a)$-factorization with $x$ factors if and only if 

\begin{tabular}{llll} 

{\rm(i)} &  $\displaystyle\frac{d+s}{r+a} \le x \le \frac{d}{r}$ & if & $r$ and $a$ are both even; \vspace{0.35cm}\\

{\rm(ii)} & $\displaystyle\frac{d+s}{r+a} < x < \frac{d}{r}$ & if & $r$ is odd and $a$ is even and $d> {\rm max}(r, r+s-a)$; \vspace{0.35cm} \\

{\rm(iii)} & $\displaystyle\frac{d+s}{r+a} < x \le \frac{d}{r}$ & if & $r$ is even and $a$ is odd and $d > r+a-s$; \vspace{0.35cm} \\

{\rm (iv)} & $\displaystyle\frac{d+s}{r+a} \le x < \frac{d}{r}$  & if & $r$ and $a$ are both odd and $d > r$.

\end{tabular}

\label{T12}\end{theorem}

\section{Preliminary Considerations}\label{sec:preliminary}
\subsection{Basic Inequalities}\label{subsec:basic}
First we show that if all simple $(d, d+s)$-graphs $G$ have an $(r, r+a)$-factorization with $x$ factors then
$$\frac{d+s}{r+a} \; \le \; x \/ \le \/ \frac{d}{r}\, .$$

\begin{lemma} Suppose that all simple $(d, d+s)$-graphs have an $(r, r+a)$-factorization.  Then
$$\frac{d+s}{r+a} \; \le \; x \; \le \; \frac{d}{r}\, .$$
\label{L13}\end{lemma}

\begin{proof} Let $G$ be a simple $(d+s)$-regular graph and suppose $x < \frac{d+s}{r+a}$.  Suppose $G$ has an $(r, r+a)$-factorization with $x$ factors. The average degree over all the factors of any vertex is $\frac{d+s}{x}$.  But $\frac{d+s}{x} > r+a$, so the largest degree of a vertex in some factor is greater than $r+a$, a contradiction.  Similarly, suppose that $G$ is a simple $d$-regular graph and that $x > \frac{d}{r}$.  The average degree over all the factors of any vertex is $\frac{d}{x} < r$, so the smallest degree in some factor is less than $r$, a contradiction.  Therefore
$$ \frac{d+s}{r+a} \; \le \; x \; \le \frac{d}{r}\;,$$
as asserted.
\end{proof}

Next we show that 
\begin{lemma} The inequalities in Theorem \ref{T12} are necessary conditions for all simple $(d, d+s)$-graphs to have an $(r, r+a)$-factorization.
\label{L14}\end{lemma}

\begin{proof} 

\begin{itemize} 
\item[(i)] follows from Lemma \ref{L13}.
\item[(ii)] Let $x = \frac{d}{r}$. First suppose that $d$ is odd. Let $D$ be a graph obtained from $K_{d+2}$ by removing a $P_3$ and $\frac{1}{2}(d-1)$ $K_2'$s, so $D$ has one vertex of degree $d-1$ and the remaining vertices have degree $d$. Let $G$ be the regular graph of degree $d$ obtained from two copies of $D$ by joining the two vertices of degree $d-1$ by an edge $e$. Then $G$ is regular of degree $d$. Since $D$ has odd order, any $r$-factor of $G$ must contain the edge $e$. Since $x \ge 2$ and since in any $(r,r+a)$-factorization of $G$ with $x$ factors, each factor must be an $r$-factor, it follows that $G$ does not have an $(r,r+a)$-factorization with $x$ factors.

Next suppose that $d$ is even. Then $K_{d+1}$ has odd order and degree $d$.  Let $A=K_{d+1}$ and $x = \frac{d}{r}$. Since $r$ is odd, $A$ has no $r$-factor, and so does not have an $(r,r+a)$-factorization with $x$ factors.

\item[(iii)] In this case $r+a$ is odd. Let $x = \frac{d+s}{r+a} \ge 2$. Suppose first that $d+s$ is odd. The argument is very like that used in (ii) when $d$ is odd. Let $D$ be a graph obtained from $K_{d+s+2}$ by removing a $P_3$ and $\frac{1}{2}(d+s-1)$ $K_2'$s, so that $D$ has one vertex of degree $d+s-1$ and the remaining vertices have degree $d+s$. Take two copies of ~$D$ and join the two vertices of degree $d+s-1$ by an edge ~$e$. Call the graph obtained this way $G$. Then $G$ is regular of degree $d+s$ and has even order. Any $(r+a)$-factor of $G$ must contain the edge $e$. Since in any $(r, r+a)$-factorization of $G$ with $x$ factors, each factor must be an $(r+a)$-factor it follows that $G$ does not have an $(r, r+a)$-factorization with $x$ factors.

If $d+s$ is even, then $K_{d+s-1}$ has odd order and even degree $d+s$. Since $r+a$ is odd, $K_{d+s-1}$ has no $(r+a)$-factor, and it follows as in Case 1 that $G$ does not have an $(r, r+a)$-factorization with $x$ factors.

\item[(iv)] First suppose that $x= \frac{d}{r}$. We assume in this case, as in case (ii), that $r$ is odd. Then the argument in case (ii) works verbatim to give examples (when $d$ is odd and when $d$ is even) of $(d, d+s)$-graphs which do not have an $(r, r+a)$-factorization with $x$ factors.

Next suppose that $x= \frac{d+s}{r+a}$. We assume in this case, as in the case (iii), that $r+a$ is odd. Then the argument in case (iii) works verbatim to give examples (when $d+s$ is odd and when $d+s$ is even) of $d, d+s)$-graphs which do not have an $(r, r+a)$-factorization with $x$ factors.
\end{itemize}
\end{proof}

In \cite{Hilton10} the following lemma was proved:
\begin{lemma} If $d,r$ and $a$ are positive integers, and $s$ is a non-negative integer, and if 
$$\frac{d+s}{r+a} \; \le \; x \; \le \; \frac{d}{r}$$
then every $(2d, 2d+2s)$-pseudograph is $(2r, 2r+2a)$-factorizable with $x$ factors. 
\label{L15} \end{lemma}

Lemma \ref{L15} was derived in a not very complicated way from a similar result for bipartite multigraphs, which is relatively straightforward to prove.

We shall use Lemma \ref{L15} to prove:

\begin{lemma} Let $r \ge 2, a \ge 2, s \ge 0$.  Then Theorem \ref{T12}(i) is true.
\label{L16} \end{lemma}

\begin{proof} The necessity follows from Lemma \ref{L14}. We shall derive the sufficiency from Lemma \ref{L15}.

\begin{itemize}
\item[(1)] If $d$ and $s$ are even this follows from Lemma \ref{L15}.

\item[(2)] If $d$ and $s$ are both odd, if
$$\frac{d+s}{r+a} \; \le \; x \; \leq \frac{d}{r}$$
then $$\frac{(d-1)+(s+1)}{r+a} \; \le \; x \le \; \frac{d-1}{r}$$
so, by Lemma \ref{L15}, every $(d-1, (d-1) + (s+1))$-pseudograph is $(r,r+a)$-factorizable with $x$ factors. But a $(d, d+s)$-pseudograph is a $(d-1, (d-1)+(s+1))$-pseudograph. Therefore every $(d,d+s)$-pseudograph is $(r,r+a)$-factorizable with $x$ factors.

\item[(3)] If $d$ is odd and $s$ is even and $\frac{d+s}{r+a} \le x \le \frac{d}{r}$, then
$$\frac{(d-1)+(s+2)}{r+a} \; \le x \; \le \; \frac{d-1}{r}\,,$$
so, by Lemma \ref{L15}, every $(d-1, (d-1) + (s+2))$-pseudograph is $(r, r+a)$-factorable with $x$ factors. But every $(d, d+s)$-pseudograph is a $(d-1, (d-1) + (s+a))$-pseudograph, so every $(d,d+s)$-pseudograph is $(r,r+a)$-factorable with $x$ factors.

\item[(4)] If $d$ is even and $s$ is odd and $\frac{d+s}{r+a} \le x \le \frac{d}{r}$, then
$$\frac{d+(s+1)}{r+a}\; \le \; x \le \; \frac{4}{r}\,,$$
so, by Lemma \ref{L15}, every $(d, d+(s+1))$-pseudograph is $(r, r+a)$-factorable with $x$ factors. But a$(d,d+s)$-pseudograph is a $(d, d+(s+1))$-pseudograph, so every $(d, d+s)$-pseudograph is $(r,r+a)$-factorable with $x$ factors.
\end{itemize}
 \end{proof}

This completes the proof of the sufficiency in Lemma \ref{L16}.

\subsection{Equitable edge-colourings}\label{subsec:equitable}
We need various results about equitable and nearly equitable edge-colourings of simple graphs.

\medskip
\noindent{\bf Definition of an equitable edge-colouring}

\noindent If $\phi: E(G) \rightarrow \mathcal{C}$, where $\mathcal{C}$ is a set of colours, then $\phi$ is {\em equitable} if $$ ||\alpha(v)| - |\beta(v)|| \; \le \; 1\; ,$$
where $\alpha(v)$ and $\beta(v)$ are the sets of edges incident with $v \in V(G)$ coloured $\alpha$ and $\beta$ respectively, for every pair $\alpha$, $\beta$ of colours of $\mathcal{C}$, and for every vertex $v \in V(G)$.

\medskip
\noindent{\bf Definition of nearly equitable edge-colouring}

\noindent This is the same as above except that the requirement is that
$$||\alpha(v)| - |\beta(v)||  \; \le \; 2\; .$$

The oldest result on this topic is due independently to McDiarmid \cite{McDiarmid72} and de Werra \cite{DeWerra71}, and is not restricted to simple graphs.

\begin{theorem} Let $x$ be a positive integer and let $G$ be a bipartite multigraph.  Then $G$ has an equitable edge-coloring with $x$ colours.
\label{T17}\end{theorem}

A result just for simple graphs was proved by Hilton and de Werra \cite{HiltonDeWerra94}.

\begin{theorem} Let $x$ be a positive integer and let $G$ be a simple graph.  Suppose that no two vertices $v$ and $w$ such that $x|d(v)$ and $x|d(w)$ are adjacent.  Then $G$ has an equitable edge-colouring with $x$ colours.
\label{T18}\end{theorem}

A nice improvement to this theorem by Xia Zhang and Guizhen Liu \cite{ZhangLiu11} appeared recently.

\vspace{0.25cm}
We can use Theorem \ref{T18} to prove the following very useful theorem.

\begin{lemma} Let $a, r$ be positive integers and $s$ a non-negative integer. If $G$ is a simple $(d, d+s)$-graph satisfying  $\frac{d+s}{r+a} < x \;< \frac{d}{r}$, then $G$ has an $(r, r+a)$-factorization with $x$ factors.
\label{L19}\end{lemma}

\begin{proof} In this case $r < \frac{d}{x} \le \frac{d+s}{x} < r+a$.  At each vertex $x$ where $x|d(v)$, it follows that $r+1 \le \frac{d(v)}{x} \le r+a - 1$.  We form a simple graph $G^+$ from $G$ by joining a pendant edge to each vertex of $G$ satisfying $x|d(v)$.  For each vertex $v$ of the simple graph $G^+$ we have $x\nmid d_{G^+}(v)$, and so $G^+$ has an equitable edge-colouring with $x$ colours, by Theorem \ref{T18}. Restricting this edge-colouring to $G$ gives an edge-colouring of $G$ which is equitable at the vertices $v$ where $x \nmid d(v)$,  and is nearly equitable at the vertices $v$ where $x|d(v)$.  Thus for each pair of colours $\alpha$ and $\beta$,
$$\begin{array}{ll}||\alpha(v)| - |\beta(v)| | \; \le \; 1 & {\rm if} \; x \nmid d(v)\,, \\ ||\alpha(v)| - |\beta(v)| | \; \le \; 2  & {\rm if} \; x\, |\,d(v)\,.\end{array}$$

The average number of edges of each colour at $v$ is exactly $\frac{d(v)}{x}$ if $x \nmid d(v)$. Then $r < \frac{d(v)}{x} < r+a$,  so $r \le \alpha(v) \le r+a$ for each colour $\alpha$.  If $x|d(v)$ then $r+1 \le \frac{d(v)}{x} \le r+a-1$, so again $r \le \alpha(v) \le r+a$ for each colour $\alpha$.  Therefore each colour class is an $(r, r+a)$-factor, and so $G$ has an $(r, r+a)$-factorization with $x$ factors.
\end{proof}

We can now prove another case when Theorem \ref{T12} is true.

\begin{lemma} Theorem \ref{T12}(ii) is true.
\label{L20}\end{lemma}

\begin{proof} In this case $r$ is odd and $a$ is even and $d > {\rm max}(r,r+s-a)$. By Lemma \ref{L14}(ii) the condition $\frac{d+s}{r+a} < x < \frac{d}{r}$ is necessary. By Lemma \ref{L19} this condition is sufficient for $G$ to have an $(r,r+a)$-factorization with $x$ factors.
\end{proof}

\section{A lower bound for $\sigma(r,s,a,t)$ (achieved when $r$ is even and $a$ is even)}\label{sec:lowerbound}

In this section we recall that if $r$ and $a$ are both even and positive then any $(d,d+s)$-simple graph has an $(r,r+a)$-factorization with $x$ factors if and only if $\frac{d+s}{r+a} \le x \le \frac{r}{a}$. This was Theorem \ref{T12}(i).

\vspace{0.25cm}
We also have from Lemma \ref{L2} that, for all $r, s, a, t$ with $r, t \ge 1$ and $a, s\ge 0$,
$$ \sigma(r,s,a,t) \ge N(r,s,a,t) \; = \; r \left\lceil \frac{tr+s-1}{a} \right\rceil + (t-1)r\,.$$
We need to show that this lower  bound for $r(r,s,a,t)$ is achieved when $r$ and $a$ are both even. It suffices to prove:

\begin{lemma} Let $r \ge 2$ and $a \ge 2$ both be even. Let $s \ge 0$ and $t \ge 1$.  Then
$$\sigma(r,s,a,t) \le r \left\lceil \frac{tr+s-1}{a} \right\rceil + (t-1)r\,.$$
\label{L21}\end{lemma}

\begin{proof}
First note that a number $p$ satisfies
$$p=r \left\lceil \frac{rt+s-1}{a}\right\rceil + (t-1)r$$
if and only if
$$p=r\left(\frac{rt+s+c}{a}\right) + (t-1)r$$
for some integer $c$ such that $a \,|\, tr+s+c$ and $-1\le c \le a-2$.

Let
$$d=\frac{r}{a} (tr+s+c)+(t-1)r+k$$
where $k \ge 0$. We show that, in this case, there do exist $t$ values of $x$ between $\frac{d+s}{r+a}$ and $\frac{d}{r}$. Then it follows from Theorem \ref{T12}(i) that every $(d, d+s)$-simple graph is $(r, r+a)$-factorable into $x$ factors for at least $t$ values of $x$.

It is easy to see that
$$\frac{d}{r} \; = \; \frac{1}{a}\, (tr+s+c)+(t-1) + \frac{k}{r}$$
and that
$$ \frac{d+s}{r+a} \; = \; \frac{1}{a}\, (tr+s+c) - \frac{r+c}{r+a} + \frac{k}{r+a}\,.$$
Therefore if $r+c \ge k \ge 0$ then, since $r+a > r+a - 2 \ge r+c$, the values of $x$ lying between $\frac{d+s}{r+a}~$ and $\frac{d}{r}$ include
$$\frac{1}{a}\,(tr+s+c), \ldots, \frac{1}{a}\, (tr+s+c) + (t-1)\,,$$
so there are at least $t$ such values of $x$.

Next suppose that $k=r+c+y$ where $(p-1)(r+a) < y \le p(r+a)$ and $p \ge 1$. Then
$$\begin{array}{lll} \displaystyle\frac{d+s}{r+a} & = & \displaystyle\frac{1}{a}\, (tr+s+c) + \frac{r+c-k}{r+a} \vspace{0.25cm}\\
& = & \displaystyle\frac{1}{a}\, (tr+s+c) + \frac{y}{r+a} \vspace{0.25cm} \\
& \le & \displaystyle\frac{1}{a}\, (tr+s+c) + p\,\end{array}$$
and
$$\begin{array}{lll} \displaystyle\frac{d}{r} & = & \displaystyle\frac{1}{a}\, (tr+s+c) + (t-1) + \frac{k}{r} \vspace{0.25cm} \\
& = & \displaystyle \frac{1}{a}\, (tr+s+c) + (t-1) + \frac{r+c+y}{r} \vspace{0.25cm} \\
& = & \displaystyle \frac{1}{a}\, (tr+s+c) + t + \frac{c+y}{r} \vspace{0.25cm} \\
& \ge & \displaystyle \frac{1}{a}\, (tr+s+c) +t+\frac{y-1}{r} \vspace{0.25cm} \\
& \ge & \displaystyle \frac{1}{a}\, (tr+s+c)+t+ \frac{(p-1)(r+a)}{r} \vspace{0.25cm} \\
& \ge & \displaystyle \frac{1}{a}\, (tr+s+c)+t+(p-1)\,.\end{array}$$
The integer values of $x$ between $\frac{d+s}{r+a}$ and $\frac{d}{r}$ are
$$\frac{1}{a} \, (tr+s+c) + (p-1) + i$$ 
for $i=1, \ldots, t$. Thus there are at least $t$ such integer values.

\vspace{0.25cm}
So indeed
$$\sigma(r,s,a,t) \le r \left\lceil \frac{tr+s-1}{a}\right\rceil + (t-1)r$$
as asserted.

\end{proof}

\section{An upper bound for $\sigma(r,s,a,t)$ (achieved when $r$ is odd and $a$ is even)}\label{sec:upperbound}

Recall that if $r \ge 1$ is odd and $a \ge 2$ is even and $d > {\rm max}(r, r+a)$, then any $(d, d+s)$-simple graph has an $(r,r+a)$-factorization with $x$ factors if and only if $\frac{d+s}{r+a} < x < \frac{d}{r}$ when $d > {\rm max}(r, r+s-a)$; this was Theorem \ref{T12}(ii).

\vspace{0.25cm}
We first prove the following upper bound for $\sigma(r,s,a,t)$, valid for all $r \ge 2$, $a \ge 2$, $t \ge 1$ and $s \ge 0$.

\begin{theorem} Let $r \ge 1$ and $a \ge 2$, $t \ge 1$ and $s \ge 0$. Then
$$\sigma(r,s,a,t) \; \le \; r \left\lceil \frac{tr+s+1}{a}\right\rceil \; + (t-1)r + 1\,.$$
\label{T22}\end{theorem}

\begin{proof}
Let us first point out that a number $p$ satisfies
$$ p \; = \; r  \left\lceil \frac{tr+s+1}{a} \right\rceil\; + (t-1)r$$
if and only if
$$p \;= \;\frac{r}{a}(tr+s+c)+(t-1)r$$
for some integer $c$ such that
$$a\;|\;tr+s+c$$
and
$$ 1 \le c \le a\,.$$

We show that
$$\sigma(r,s,a,t) \le r \left\lceil \frac{tr+s+1}{a}\right\rceil + (t-1) r+1\,.$$
So we show that
$$\sigma(r,s,a,t) \le \frac{r}{a}(tr+s+c) + (t-1) r+1\,,$$
where $a \; | \; tr+s+c$ when $1 \le c \le a$.

Let $$d \; = \; \frac{r}{a}(tr+s+c)+(t-1)r+k$$
where $k \ge 1$. We show that in this case there do exist at least $t$ integer values of $x$ satisfying $\frac{d+s}{r+a} < x < \frac{d}{r}$. Then it follows by Theorem \ref{T12}(ii) that every $(d, d+s)$-simple graph is $(r,r+a)$-factorable into $x$ factors for at least $t$ values of $x$.

Note that
$$\frac{d}{r} \; = \; \frac{1}{a}(tr+s+c)+(t-1)+\frac{k}{r}$$
and that
$$\begin{array}{lll}d+s & = & \displaystyle \frac{r}{a} (tr+s+c)+(t-1)r+k+s \vspace{0.25cm}\\
& = & \displaystyle \frac{r+a}{a} (tr+s+c)+(t-1)r+k+s-(tr+s+c) \vspace{0.25cm} \\
& = & \displaystyle \frac{r+a}{a}(tr+s+c)-(r+c)+k\,,\end{array}$$
so that
$$\frac{d+s}{r+a} \; = \; \frac{1}{a} (tr+s+c)- \frac{r+c}{r+a} + \frac{k}{r+a}\,.$$

If $1 \le k \le r+c-1$ then the integer values of $x$ satisfying $\frac{d+s}{r+a} < x < \frac{d}{r}$ include
$$ \frac{1}{a} (tr+s +c) + i$$
for $i = 0,1, \ldots, t-1$. Thus there are at least $t$ such integer values of $x$ in this case.

For $r+c=k$ then the integer values of $x$ satisfying $\frac{d+s}{r+a} < x < \frac{d}{r}$ include
$$\frac{1}{a} tr +s+c) +1$$
for $i = 1, 2, \ldots, t$. thus there are at least $t$ such integer values of $x$ in this case.

For $r+c<k$ then the integer values of $x$ satisfying $\frac{d+s}{r+a} < x < \frac{d}{r}$ include
$$\frac{1}{a}(tr+s+c)+i$$
for $i=\left\lfloor\frac{k-r-c}{r+a}\right\rfloor + 1, \ldots, \left\lfloor \frac{k-r-c}{r+a} \right\rfloor + (t-1) + 1$, since

$$\frac{1}{a}(tr+s+c)+\left\lfloor \frac{k-r-c}{r+a}\right\rfloor + t < \frac{d}{r}$$
i.e., $\left\lfloor \frac{k-r-c}{r+a}\right\rfloor + 1 < \frac{k}{r}$, which is true since
$$\left\lfloor \frac{k-r-c}{r+a}\right\rfloor \; \le \; \frac{k-r-c}{r+a} \; < \; \frac{k-r}{r} \; ,$$
so there are at least $t$ such integer values of $i$.

Therefore,
$$\sigma(r,s,a,t) \; \le \; r \left\lceil \frac{tr+s+a}{a}\right\rceil + (t-1)r+1\,.$$

\end{proof}

In the case when $r$ is odd and $a$ is even, so that by Theorem \ref{T12}(ii)
$$\frac{d+s}{r+a} \; < x \; < \frac{d}{r}\;,$$
the upper bound in Theorem \ref{T22} is achieved as we now show.

\begin{theorem} Let $r \ge 1$ be odd, $a \ge 2$ be even. Let $t \ge 1$ and $s \ge 0$. Then
$$\sigma(r,s,a,t) \; = \; \left\{\begin{array}{ll} r\left\lceil \frac{tr+s+1}{a}\right\rceil + (t-1)r+1 & {\rm if}\; t \ge 2 \; {\rm or} \; {\rm if} \; t=1 \; \\
& {\rm and}\; a<r+s+1 \vspace{0.25cm} \\
r & {\rm if} \; t=1 \; {\rm and} \; a \ge r+s+1\,.\end{array} \right.$$
\label{T23}\end{theorem}

\begin{proof}
From Theorem \ref{T22} we already know that
$$\sigma(r,s,a,t) \; \le \; r \left\lceil \frac{tr+s+1}{a}\right\rceil + (t-1) r+1\,.$$
First assume that $t \ge 2$, or $t=1$ and $a < r+s+1$.  The equation
$$p \; = \; r \left\lceil \frac{tr+s+1}{a} \right\rceil + (t-1)r+1$$
is true if and only if
$$ p \; = \; r \left(\frac{tr+s+c}{a}\right) + (t-1)r+1$$
for some integer $c$ such that $a \; | \; tr+s+c$ and $1 \le c \le a$.

We first show that if
$$d \; = \; r \left\lceil \frac{tr+s+1}{a} \right\rceil + (t-1)r$$
and $d \ge {\rm max} \{r, r+a - s\}$ then there is an example of a $(d,d+s)$-simple graph $G$ which does not have an $(r,r+a)$-factorization with $x$ factors for $t$ different values of $x$. It suffices to show that there do not exist $t$ integer values of $x$ satisfying
$$ \frac{d+s}{r+a}\; < \; x \; < \; \frac{d}{r}\,.$$
So suppose that
$$d\; = \; \frac{r}{a}(tr+s+c)+(t-1)r$$
where $a \, | \, tr+s+c$ and $1 \ge c \ge a$.  Then
$$\frac{d}{r} \; = \; \frac{1}{a}(tr+s+c)+(t-1)$$
and
$$\begin{array}{lll}d+s & = & \displaystyle\frac{r}{a}\, (tr+s+c)+(t-1)r+s \vspace{0.25cm} \\
& = & \displaystyle \frac{r+a}{a}\, (tr+s+c)+(t-1)r+s-(tr+s+c) \vspace{0.25cm} \\
& = & \displaystyle \frac{r+a}{a}\, (tr+s+c)-(r+c) \end{array}$$
so that
$$\frac{d+s}{r+a} \; = \; \frac{1}{a}\, (tr+s+c)- \frac{r+c}{r+a} \,.$$
The integer values of $x$ which satisfy $\frac{d+s}{r+a} < x < \frac{d}{r}$ are
$$\frac{1}{a}\, (tr+s+c)+ i $$
for $i=0, 1, \ldots, t-2$, giving only $t-1$ values altogether. Therefore
$$\sigma(r,s,a,t) \ge \frac{r}{a}\, (tr+s+c)+(t-1)r+1$$
where $a \, | \,tr+s+c$ and $ 1 \le c \le a$. In other words
$$\sigma(r,s,a,t) \, \ge \, r \left\lceil\frac{tr+s+1}{a}\right\rceil + (t-1)r+1\,.$$
In view of Theorem \ref{T22}, it now follows that
$$\sigma(r,s,a,t) \; = \; r \left\lceil\frac{tr+s+1}{a} \right\rceil + (t-1)r+1\,.$$

Secondly suppose that $t=1$ and that $a \ge r+s+1$. Then Theorem \ref{T22} tells us that $\sigma(r,s,a,t) \le r+1$ in this case. But if $d=r$ then a $(d,d+s)$-graph is already an $(r,r+a)$ factor, so that $\sigma(r,s,a,t)=r$.

\end{proof}

\section{Some preliminary remarks before the two remaining cases (where $a$ is odd)}\label{sec:preliminaryremarks}

From Lemma \ref{L2} and Theorem \ref{T22} it follows that for all $r, s, a, t$ with $r \ge 1$, $a \ge 2$, $t \ge 1$, $s \ge 0$
$$ r \left\lceil \frac{rt+s-1}{a} \right\rceil + (t-1)r \; \le \; \sigma(r,s,a,t) \; \le \; r \left\lceil\frac{rt+s+1}{a} \right\rceil + (t-1) r + 1\,.$$
Thus $\sigma(r,s,a,t)$ is already tightly bounded. We also note that we have not proved so far Theorem \ref{T12}(iii) and Theorem \ref{T12}(iv) which would tell us that
$$\frac{d+s}{r+a} \; < \; x \; \le \; \frac{d}{r} \;\;\;\; {\rm if} \; r \; {\rm is \; even \; and} \; a \; {\rm is\; odd \; and} \; d > {\rm max}(r,r+s-a)\,,$$
and
$$\frac{d+s}{r+a} \; \le \; x \; < \; \frac{d}{r} \;\;\;\;{\rm if} \; r \; {\rm and} \;a\; {\rm are \; both\; odd \; and}\; d>r\,. \hspace{3.2cm}$$
It seems to be quite hard to provide a direct proof of these inequalities, from which the bounds in \ref{T11}(iii) and \ref{T11}(iv) would follow by the same arguments as were used in Section \ref{sec:upperbound}. Instead we have it seems to finesse \ref{T12}(iii) and \ref{T12}(iv) by what might seem to be rather roundabout arguments.

\section{The threshold number $\sigma(r,s,a,t)$ when $r$ is even and $a$ is odd}\label{sec:threshold}

In this section we prove:

\vspace{0.25cm}
\noindent{\bf Theorem \ref{T12}(iii)}.  Let $r \ge 2$ be even and $a \ge 1$ be odd. Then every $(d,d+s)$-simple graph $G$ has an $(r,r+a)$-factorization with $x$ factors if and only if
$$\frac{d+s}{r+a} \; < \; x \; \le \; \frac{d}{r}\,.$$
Note that in Lemma \ref{L14} we proved the necessity of this condition; namely we showed that if $r$ is even and $a$ is odd, and if every $(d, d+s)$-simple graph has an $(r,r+a)$-factorization with $x$ factors, then $x$ satisfies the inequality above. So it remains to prove the sufficiency.

We also prove in this section:

\vspace{0.25cm}
\noindent{\bf Theorem \ref{T11}(iii)}. Let $ r \ge 2$ be even and $ a \ge 1$ be odd.

\noindent Then
$$\sigma(r,s,a,t) \; = \; r \left\lceil \frac{rt + s}{a} \right\rceil + (t-1)r\,.$$

We start by improving very slightly the lower bound for $ \sigma(r,s,a,t)$ given by Lemma \ref{L2}. We prove

\begin{theorem} Let $r$ be even, $r \ge 2$, and $a \ge 1$ be odd. Let $t$ be a positive integer and $s$ a non-negative integer. Then
$$ \sigma(r,s,a,t) \; \ge \; r \left\lceil\frac{tr+s}{a} \right\rceil + (t-1)r\,.$$
\label{T24}\end{theorem}

\begin{proof}
First let us remark that a number $p$ satisfies
$$p\; = \; r \left\lceil \frac{tr+s}{a}\right\rceil + (t-1)r$$
if and only if
$$p \; = \; \frac{r}{a}(tr+s+c)+(t-1)r$$
for some integer $c$ such that
$$a \, | \, tr+s+c$$
and
$$0 \le c \le a - 1\,.$$

Suppose that an integer $d$ satisfies
$$d \; = \; \frac{r}{a}(tr+s+c) + (t-1)r-1$$
where
$$a \, |\, tr+s+c$$
and
$$0 \le c \le a - 1\,.$$
Then
$$\frac{d}{r} \; = \; \frac{1}{a}(tr+s+c)+(t-1) - \frac{1}{r}$$
and
$$\frac{d+s}{r+a} \; = \frac{1}{a}(tr+s+c) - \frac{r+c+1}{r+a}\,,$$
since
$$\begin{array}{lll} d+s & = & \displaystyle \frac{r+a}{a}\,(tr+s+c)+(t-1)r-1-(tr+s+c)+s \vspace {0.25cm}\\
& = & \displaystyle\frac{r+a}{a}\,(tr+s+c)-r-c-1\,. \end{array}$$

Since $c+a \le a$ it follows that the integer values of $x$ which satisfy $\frac{d+s}{r+a} < x \le \frac{d}{r}$ are
$$\frac{1}{a}(tr+s+c)+j$$
for $j=0, 1, \ldots, t-2$, so there are fewer than $t$ such values of $x$. So it follows that if there are at least $t$ such values of $x$ then
$$d \ge \frac{r}{a}(tr+s+c)+(t-1)r\,,$$
so that
$$ d\ge \left\lceil \frac{tr+s}{a}\right\rceil + (t-1)r\,.$$
Consequently
$$\sigma(r,s,a,t) \ge r \left\lceil \frac{tr+s}{a}\right\rceil + (t-1)r$$
when $r$ is even and $a$ is odd.

\end{proof}

Next we lower the upper bound for $\sigma(r,s,a,t)$ obtained in Theorem \ref{T22}, this lowering being valid for the case when $r \ge 2$ even, $a \ge 1$ odd. We also show that there are $t$ values of $x$ satisfying $\frac{d+s}{r+a} < x \le \frac{d}{r}$.

\begin{theorem} Let $r$ be even, $r \ge 2$, and $ a \ge 1$ be odd. Let $t$ be a positive integer and $s$ a non-negative integer. Then
$$r \left\lceil \frac{tr+s}{a}\right\rceil + (t-1)r \; \le \; \sigma (r,s,a,t) \; \le \; r \left\lceil \frac{tr+s+1}{a}\right\rceil + (t-1)r\,;$$
furthermore if $d \ge \sigma(r,s,a,t)$ then there are $t$ values of $x$ satisfying $\frac{d+s}{r+a} < x \le \frac{d}{r}$.
\label{T25}\end{theorem}

\begin{proof} The earlier upper bound was established in Theorem \ref{T22}. We make progress by examining the proof of Theorem \ref{T22} in  more detail.

We assumed that $d = \frac{r}{a} (tr+s+c) + (t-1)r + k$, where $a \, | \, tr+s+c$ and $1 \le c \le k$, and $k \ge 1$. Then 
$$\frac{d}{r} \; = \; \frac{1}{a} (tr+s+c) + (t-1)+\frac{k}{r}$$
and
$$\frac{d+s}{r+a} \; = \; \frac{1}{a}(tr+s+c) - \frac{r+c}{r+a} + \frac{k}{r+a}\,.$$
Then, for $k \ge 1$, the number of values of $x$ satisfying
$$\frac{d+s}{r+a} \; < \; x \; < \; \frac{d}{r}$$
is at least $t$. If $k=0$ there are only $t-1$ such values of $x$. But in this case
$$\frac{d}{r} \; = \; \frac{1}{a}(tr+s+c)+(t-1)$$
and
$$\frac{d+s}{r+a} \; = \; \frac{1}{a}(tr+s+c) - \frac{r+c}{r+a}\,,$$
and the values of $x$ satisfying
$$\frac{d+s}{r+a} \; < \; x \; \le \; \frac{d}{r}$$
(with $\frac{d}{r} = x$ now being allowed) are 
$$\frac{1}{a}\,(tr+s+c) + i \quad {\rm for} \;\; i=0, 1, \ldots, t-1$$
so there are $t$ values of $x$ in this case. Thus in every case, there are at least $t$ values of $x$ satisfying
$$\frac{d+s}{r+a} \; < x \; \le \frac{d}{r}\,.$$
It follows that
$$\sigma(r,s,a,t) \; \le \; r \left\lceil \frac{tr+s+1}{a}\right\rceil + (t-1)r\,.$$
\end{proof}

In Theorem \ref{T25} we showed that if $G$ is a $(d,d+s)$-simple graph with $d \ge r \left\lceil\frac{tr+s}{a}\right\rceil + (t-1)r$ then at least $t$ values of $x$ satisfy $\frac{d+s}{r+a} < x \le \frac{d}{r}$. In particular, if $d \ge r \left\lceil \frac{r+s}{a}\right\rceil$, then every $(d,d+s)$-simple graph $G$ has an $(r,r+a)$-factorization with $x$ factors if $\frac{d+s}{r+a} < x \le \frac{d}{r}$ in the case when $r \ge 2$ is even and $a \ge 1$ is odd. Taken together with the necessity part of Theorem \ref{T12}(iii), and Lemma \ref{L14} this proves:

{\bf Theorem \ref{T12}(iii)}. Let $ r\ge 2$ be even, $a \ge 1$ be odd, and let $ s \ge 0$. Then every $(d,d+s)$-simple graph $G$ has an $(r,r+a)$-factorization with $x$ factors, where $x$ is an integer, if and only if
$$\frac{d+s}{r+a} \; < \; x \; \le \; \frac{d}{r}\,.$$

We finally turn to the proof of the equality
$$\sigma(r,s,a,t)\; =\; r \left\lceil \frac{tr+s}{a} \right\rceil + (t-1) r$$
when $r \ge 2$ is even and $a \ge 1$ is odd. There is more than one way of proving this at this point, but we want to show that Theorem \ref{T12}(iii) implies Theorem \ref{T11}(iii).

\begin{theorem} Let $r$ be even, $r \ge 2$, and let $a \ge 1$ be odd. Let $t \ge 1$ and $s \ge 0$ be integers. Then 
$$\sigma(r,s,a,t) \; = r \left\lceil \frac{tr+s}{a}\right\rceil + (t-1)r\,.$$
\label{T26} \end{theorem}

\begin{proof}
By Theorem \ref{T25}, 
$$ \sigma(r,s,a,t) \; \ge \; \left\lceil \frac{tr+s}{a}\right\rceil + (t-1)r\,.$$
So we need to show that
$$\sigma(r,s,a,t) \; \le \; \frac{r}{a}(tr+s+c) + (t-1)r\,,$$
where, as in the proof of Theorem \ref{T24}, 
$$ a\, | \, tr+s+c$$
and
$$ 0 \le c \le a - 1\,.$$

Let
$$ d\; = \; \frac{r}{a}(tr+s+c)+(t-1)r+k$$
where $k \ge 0$.  We show that there exist $t$ integer values of $x$ satisfying
$$\frac{d+s}{r+a} \; < x \; \le \frac{d}{r}\,.$$
Then it follows by the definition of $\sigma(r,s,a,t)$ that every $(d,d+s)$-simple graph is $(r,r+a)$-factorable into $x$ factors for at least $t$ integer values of $x$.

\vspace{0.25cm}
First we note that 
$$\frac{d}{r} \; = \; \frac{1}{a}(tr+s+c)+(t-1)+\frac{k}{r}$$
and
$$\frac{d+s}{r+a} \; = \; \frac{1}{a}(tr+s+c)+\frac{k-r-c}{r+a}\;.$$

\vspace{0.25cm}
\noindent For $p$ a non-negative integer, if $pr \le k < (p+1)r$ then $\frac{k}{r} \ge p$ and
$$\frac{k-r-c}{r+a} \; < \; \frac{(p+1)r-r-c}{r+a} \;= \; \frac{pr-c}{r+a} \; \le \; p\frac{r}{r+a} \; < \; p$$
so
$$\frac{d+s}{r+a} \; = \; \frac{1}{a}(tr+s+c) + \frac{k-r-c}{r+a} \; < \; \frac{1}{a}(tr+s+c)+p $$
and
$$\begin{array}{lll}\displaystyle\frac{d}{r} & = & \displaystyle\frac{1}{a}(tr+s+c)+(t-1)+\frac{k}{r} \vspace{0.25cm}\\
& \ge & \displaystyle\frac{1}{a}(tr+s+c)+(t-1)+p\,.\end{array}$$

\vspace{0.2cm}
\noindent Therefore if $pr \le k < (p+1)r$ for some non-negative integer $p$, then the integer values of $x$ satisfying
$$ \frac{d+s}{r+a} \; < \; x \; \le \frac{d}{r}$$
include
$$\frac{1}{a}(tr+s+c)+i$$
for $i = p, p+1, \ldots, p+(t-1)$ so there are at least $t$ such values of $x$. Therefore
$$\sigma(r,s,a,t) \; \le \; r \left\lceil\frac{tr+s}{a}\right\rceil + (t-1)r$$
as asserted. Theorem \ref{T26} now follows.
\end{proof}

\section{The threshold number $\sigma(r,s,a,t)$ when $r$ is odd and $a$ is odd}\label{sec:thresholdnumber}
The discussion in this section is rather like the discussion in the previous section, but it is sufficiently different, that, for the sake of clarity, we need to treat it separately.

We shall prove:

\vspace{0.25cm}
\noindent{\bf Theorem \ref{T12}(iv)}.  Let $r \ge 1$ be odd and $a \ge 1$ be odd, and let $d > r$.  Then every $(d,d+s)$-simple graph $G$ has an $(r,r+a)$-factorization with $x$ factors if and only if 
$$\frac{d+s}{r+a} \; \le \; x \; < \; \frac{d}{r}\,.$$
In Lemma \ref{L14} we proved the necessity of this condition. So it remains to prove the sufficiency.

\vspace{0.25cm}
\noindent We also prove:

\vspace{0.25cm}
\noindent{\bf Theorem \ref{T11}(iv)}.  Let $r \ge 1$ be odd and $a \ge 1$ be odd. Let $s \ge 0$ and $t \ge 1$. Then
$$\sigma(r,s,a,t) \; = \; \left\{ \begin{array}{l}  r\left\lceil\displaystyle\frac{tr+s}{a}\right\rceil + (t-1)r+1 \quad {\rm if} \;\; t \ge 2 \;\; {\rm or} \;\;{\rm if} \;\;t=1  \\
 \hspace{5cm}{\rm and} \;\; a < r+s\,, \vspace{0.25cm}\\
r \;\;\; {\rm if} \; \; t=1 \;\; {\rm and} \;\; a \ge r+s\,.\end{array}\right.$$

\vspace{0.25cm}
\noindent We first prove the following theorem, which gives a lower bound for $\sigma(r,s,a,t)$ in this case.

\begin{theorem} Let $r \ge 1$ be odd, $a \ge 1$ be odd, $t \ge 1$ and $s \ge 0$ be integers. Then
$$\sigma(r,s,a,t) \; \ge \; \left\{ \begin{array}{l}  r\left\lceil\displaystyle\frac{rt+s}{a}\right\rceil + (t-1)r+1 \quad {\rm if} \;\; t \ge 2 \;\; {\rm or} \;\; t=1 \;\; \\ \hspace{5cm} {\rm and} \;\; a<tr+s\,, \vspace{0.25cm}\\
 r  \;\;\; {\rm if} \;\; t=1 \;\;  {\rm and} \;\; a \ge r+s\,.\end{array}\right.$$
\label{T27}\end{theorem}

\begin{proof}
First suppose that $t \ge 2$ or $t=1$ and $a < r+s$.  Let us remark that an integer $p$ satisfies
$$p\; = \; r \left\lceil\frac{rt+s}{a}\right\rceil + (t-1)r+1$$
if and only if
$$p\; = \; r\left(\frac{rt+s+c}{a}\right) + (t-1)r+1$$
where $a \, | \, tr+s+c$ and $0 \le c \le a-1$.

\vspace{0.25cm}
Let an integer $d$ satisfy
$$d \; = \; \frac{r}{a}(tr+s+c)+(t-1)r$$
for some $c$ such that $a \,|\, tr+s+c$ and $0 \le c \le a-1$. Then
$$\frac{d}{r} \; = \; \frac{1}{a}(tr+s+c)+(t-1)$$
and
$$\frac{d+s}{r+a} \; = \; \frac{1}{a}(tr+s+c) - \frac{r+c}{r+a}\;.$$

\vspace{0.25cm}
\noindent The integer values of $x$ satisfying
$$\frac{d+s}{r+a} \; \le \; x \; \le \frac{d}{r}$$
include
$$\frac{tr+s+c}{a} + i$$
for $i = 0,1, \ldots, t-2$, since $0 \le c \le a-1$. They do not include $i=-1$ or $t-1$ or any other integer values, so there are only $t-1$ such integer values of $x$. Therefore
$$\sigma(r,s,a,t) \; > \;r\left(\frac{tr+s+c}{a}\right) + (t-1)r\,,$$
and so
$$\sigma(r,s,a,t) \; \ge \; r \left\lceil\frac{rt+s}{a}\right\rceil + (t-1)r+1\,.$$
\end{proof}

Now suppose that $t=1$ and $a > r+s$. If applied in this case, the inequality derived in the other case would (erroneously) say that $\sigma(r,s,a,t)\ge r+1$. But if $d=r$ and $G$ is an $(r,r+a)$-graph, then $G$ would be a $(d,d+s)$-graph with an $(r,r+a)$-factorization with $1$ factor. Therefore, in this case,
$$\sigma(r,s,a,t) \; \ge \; r\,.$$

Next we provide quite good bounds for $\sigma(r,s,a,t)$ when $r$ and $a$ are both odd, and also show that if  $$ d \ge r \left\lceil\frac{rt+s}{a}\right\rceil + (t-1)r+1 $$ in this case, then there are $t$ integer values of $x$ satisfying $\displaystyle\frac{d+s}{r+a} \le x \; < \; \frac{d}{r}$.

\begin{theorem} Let $ r \ge 1$ be odd, $a \ge 1$ be odd and $t=2$ or $t=1$ and $ a < rt+s$. Then
$$r \left\lceil \frac{tr+s}{a}\right\rceil + (t-1)r+1 \; \le \; \sigma(r,s,a,t) \; \le \; r \left\lceil\frac{tr+s+1}{a}\right\lceil + (t-1)r+1\,.$$
Moreover, if $d \displaystyle\ge r \left\lceil\frac{rt+s}{a}\right\rceil +(t-1)r+1$ then there are $t$ values of $x$ satisfying $\displaystyle\frac{d+s}{r+a} \le x < \frac{d}{r}$.
\label{T28}\end{theorem}

\vspace{0.25cm}
\begin{proof}
From Theorem \ref{T22} and Theorem \ref{T27}, if $ r\ge1$ and $a\ge 3$, then
$$r \left\lceil\frac{rt+s}{a}\right\rceil + (t-1)r + 1 \/ \le \/ \sigma(r,s,a,t) \; \le \; r \left\lceil\frac{rt+s+1}{a}\right\rceil + (t-1)r+1\,.$$
We know from the proof of Theorem \ref{T22} that if
$$d=r \left\lceil\frac{tr+s+1}{a}\right\rceil + (t-1)r+k$$
where $k \ge 1$, then there are at least $t$ values of $x$ satisfying
$$\frac{d+s}{r+a} \; < \; x \; < \; \frac{d}{r} \;.$$
With this value of $d$, we know that
$$d\; = \; \frac{r(rt+s+c)}{a} + (t-1)r+k$$
for some $c$, $0 < c \le a$, and $a \, |\, rt+s+c$.

\vspace{0.25cm}
But if $d = r \displaystyle\left\lceil\frac{tr+s}{a}\right\rceil + (t-1)r+k$ where $k \ge 1$  then if $a =c$, i.e. $a \, |\, rt+s$, i.e. $\displaystyle\left\lceil\frac{rt+s}{a}\right\rceil \; \neq \; \left\lceil\frac{rt+s+1}{a}\right\rceil$ it is no longer true that
$$ d \; = \; \frac{r(rt+s+c)}{a} \; + \; (t-1)r+k\,.$$
So if $k=0$ and $c=a$ then it is not true that there are $t$ values of $x$ satisfying $\frac{d+s}{r+a} < x < \frac{d}{r}$.

\vspace{0.25cm}
But there are $t$ values of $x$ satisfying $\frac{d+s}{r+a} \le x < \frac{d}{r}$. For then (following the discussion of Theorem \ref{T22}), 
$$\frac{d}{r} \; = \; \frac{1}{a}(tr+s+a) + (t-1) \; = \; \frac{1}{a}(tr + s) + t$$
and
$$\frac{d+s}{r+a} \; = \; \frac{1}{a}(tr+s+a) - \frac{r+a}{r+a} \; = \; \frac{1}{a} (tr+s)\,.$$

\vspace{0.25cm}
So the values of $x$ satisfying $\frac{d+s}{r+a} \le x < \frac{d}{r}$ are
$$\frac{1}{a}(tr+s) + i \;\;\; {\rm for} \; \; i=0,1, \ldots, t-1$$
so there are $t$ integer values of $x$ as asserted.  In that case
$$\begin{array}{lll}  d & = & \displaystyle\frac{r}{a} \left(\frac{rt+s+a}{a}\right) + (t-1)r + 0 \vspace{0.25cm}\\
& = & \displaystyle\frac{r}{a}\left(\frac{rt+s+0}{a}\right) + (t-1)r + 1 \vspace{0.25cm}\\
& = & r \displaystyle\left\lceil\frac{rt+s}{a}\right\rceil + (t-1)r+1\,. \end{array}$$
Therefore if $d \ge r\left\lceil\frac{rt+s}{a}\right\rceil + (t-1)r+1$ we have
$$\frac{d+s}{r+a} \; \le \; x \; < \; \frac{d}{r}$$
for $t$ integer values of $x$.

\end{proof}

In Theorem \ref{T28} we showed that if $r\ge1$ is odd and $a \ge 1$ is odd then for every $(d,d+s)$ simple graph with $d \ge r\left\lceil\frac{tr+s}{a}\right\rceil + (t-1)r+1$ there are at least $t$ values of $x$, satisfying $\frac{d+s}{r+a} \le x < \frac{d}{r}$. In particular, if $d \ge r \left\lceil\frac{r+s}{a}\right\rceil +1$ then every $(d, d+s)$-simple graph $G$ has an $(r,r+a)$-factorization with $x$ factors for each value of $x$ satisfying $\frac{d+s}{r+a} \le x < \frac{d}{r}$, provided $r \ge 3$ is odd and $a \ge 3$ is odd. Taken together with the necessity part of Theorem \ref{T12}(iv) this proves:

\vspace{0.25cm}
\noindent{\bf Theorem \ref{T12}(iv)}.   Let $r \ge 1$ be odd and $ a \ge 1$ be odd, and let $s \ge 0$. Then every $(d,d+s)$-simple graph $G$ has an $(r,r+a)$-factorization with $x$ factors, where $x$ is an integer, if and only if $\frac{d+s}{r+a} \; \le \; x \; < \; \frac{d}{r}$.

\vspace{0.25cm}
We finally turn to the proof of our main result in this section.

\begin{theorem} Let $r \ge 1$ be odd and $a \ge 1$ be odd. Let $s \ge 0$ and $t \ge 1$. Then
$$\sigma(r,s,a,t) \; = \; r\left\lceil\frac{rt+s}{a}\right\rceil + (t-1)r+1$$
if $t \ge 1$, or if $t=1$ and $a<rt+s$. If $t=1$ and $a \ge rt+s$ then
$$\sigma(r,s,a,t) \; = \; r\,.$$
\label{T29}\end{theorem}
We shall show that Theorem \ref{T12}(iv) implies Theorem \ref{T28}.

\begin{proof}
Let
$$d \; = \; \frac{r}{a}(tr+s+c)+(t-1)r+k$$
where $k \ge 1$. We show that there do exist $t$ integer values of $x$ satisfying $\frac{d+s}{r+a} \le x < \frac{d}{r}$. Then it follows by Theorem \ref{T12}(iv) that every $(d,d+s)$-simple graph is $(r,r+a)$-factorable into $x$ factors for at least $t$ values of $x$.

\vspace{0.25cm}
First we note that
$$\frac{d}{r} \; = \; \frac{1}{a}(tr+s+c)+(t-1)+\frac{k}{r}$$
and
$$\frac{d+s}{r+a} \; = \; \frac{1}{a}(tr+s+c)+\frac{k-r-c}{r+a}\;.$$
For $p$ a non-negative integer, if $pr< k \le (p+1)r$ then $\frac{k}{r} > p$ and
$$\frac{k-r-c}{r+a} \; \le \; \frac{(p+1)r-r-c}{r+a} \; = \; \frac{pr-c}{r+a} \; \le \; p \frac{r}{r+a} \; < \; p\,,$$
so 
$$\frac{d+s}{r+a} \; = \frac{1}{a}(tr+s+c)+\frac{k-r-c}{r+a} \; \le \; \frac{1}{a}(tr+s + c)+p$$
and
$$\begin{array}{lll}\displaystyle\frac{d}{r} & = & \displaystyle\frac{1}{a}(tr+s+c)+(t-1)+\frac{k}{r} \vspace{0.25cm}\\
& > & \displaystyle\frac{1}{a}(tr+s+c)+(t-1)+p\,.\end{array}$$
Therefore if $pr<k\le(p+1)r$ for some non-negative integer $p$, then the integer values of $x$ satisfying $\frac{d+s}{r+a} \le x < \frac{d}{r}$ include
$$\frac{1}{a}(tr+s+c)+i \;\;\;{\rm for} \; \; i = p, p+1, \ldots, p+(t-1),$$
so there are at least $t$ such values of $x$.

Therefore
$$\sigma(r,s,a,t) \; \le \; r \left\lceil\frac{tr+s}{a}\right\rceil + (t-1)r+1\,$$
as asserted. In view of Theorem \ref{T27}, the main part of Theorem \ref{T29} now follows.

\vspace{0.25cm}
If $t=1$ and $a \ge tr+s$ then the formula $r \left\lceil\frac{tr+s}{a}\right\rceil + (t-1)r+1$ yields the value $r+1$.  But when $d=r$ then, since $s<r+s\le a$, a $(d,d+s)$-graph is an $(r,r+a)$-factor. Therefore, in this case,
$$\sigma(r,s,a,t) = r\,.$$
\end{proof}

This completes the proof of Theorem \ref{T29}.

\section{Boundary graphs}\label{boundarygraphs}

From Theorems \ref{T12}(iii) and \ref{T12}(iv) we know that if $r$ is even and $a$ is odd, or if $r$ is odd and $a$ is odd, and if $\frac{d+s}{r+a} = x = \frac{d}{r}$, where $x$ is an integer, $x \ge 1$, and if there are $(d, d+s)$-graphs $G$ which satisfy this equation, then the graphs cannot have an $(r,r+a)$-factorization with $x$ factors. But do such graphs exist, and what properties do they have? By a result of Kano and Saito \cite{KanoSaito}, such graphs do have at least one $(r,r+a)$-factor. It would be natural to suppose that they have $x-1$ edge-disjoint $(r,r+a)$-factors, but we have not investigated this. In this section we give examples of such graphs.

Given positive integers $r,s,a,d$ call a graph $G$ satisfying $\frac{d+s}{r+a} = x = \frac{d}{r}$, where $x \ge 1$ is an integer, a {\em boundary graph}. Let $BG(r,s,a,d)$ be the set of all boundary graphs with parameters $r, s, a, d$.

\begin{theorem}
Let $r \ge 2$ be an even and $a \ge 3$ be an odd positive integer. Let $d$, $s$ and $x$ be positive integers such that $\frac{d+s}{r+a} = x = \frac{d}{r}$, $x \ge 2$ and $d > \frac{a}{r}$. Then $B(r,s,a,d) \neq \phi$.
\label{T30}\end{theorem}

\begin{proof}
We separate the cases $x$ even and $x$ odd. Although these are similar, it is easier for the reader if they are treated separately.

\vspace{0.25cm}
\noindent{\bf Case 1: Let $x$ be even.}

Let $G$ be a bi-degreed simple graph with vertex sets $M$ and $N$, where $|M| = xr+1$ and $|N| = x(r+a)$. Since $d > \frac{a}{r}$ it follows that $\left(\frac{xr+1}{2}\right) > \frac{x(r+a)}{2}$, so there is a simple graph $H$ with $V(H) \subseteq M$ and $|E(H)| = \frac{x(r+a)}{2}$. Label the vertices of $H$ with labels $a_1, a_2, \ldots, ..., a_{x(r+a)}$ in such a way that if $v \in V(H)$ then $v$ receives $d_H(v)$ labels. Also assign the labels $a_1, a_2, \ldots, a_{x(r+a)}$ to the vertices of $N$, assigning one label to each vertex.

We have $H$ placed on the vertices of $M$. Then, to form $G$ from this, join each vertex $v$ of $N$ to each vertex of $M$ except the vertex with the same label as $v$. Then, for $v \in N$, $d_G(v) = xr$ and, for $v \in M$, $d_G(v) = x(r+a)$.

Notice that
$$|E(G)| \; = \; xr(x(r+a)) + \frac{x(r+a)}{2} \; = \; x^2r^2 + x^2ra + \frac{x(r+a)}{2}\;.$$
If $G$ has an $(r,r+a)$-factorization with $x$ factors, let $\{F_1, \ldots, F_x\}$ be such a set of factors. Each $F_i$ will have $r+a$ edges incident with each vertex of $M$, so for $ 1 \le i \le x$,
$$\sum_{v \in V(G)} d_{F_1}(v) \ge (xr+1)(r+a) + x(r+a)r = 2xr^2 + 2 xra + r + a\,.$$
Therefore
$$ |E(F_i)| \; \ge \; xr^2  + xra + \left\lceil \frac{r+a}{2}\right\rceil\,.$$
Consequently we have
$$\begin{array}{lll} |E(G)| & = & \displaystyle\sum_{i=1}^x |E(F_i)| \vspace{0.25cm} \\
& \ge & \displaystyle x^2r^2 + x^2 ra + x \left\lceil\frac{r+a}{2}\right\rceil \vspace{0.25cm} \\
& > & x^2r^2 + x^2ra+ \displaystyle \frac{x(r+a)}{2}\;, \;\;\; {\rm since} \; r+a \; {\rm is \; odd} \vspace{0.25cm} \\
& = & |E(G)|\,,\end{array}$$
a contradiction.

Therefore $G$ does not have an $(r,r+a)$-factorization into $x$ factors when $x$ is even.

In Figure 1 we give an example which illustrates the construction used in Theorem \ref{T30}, Case 1. Here $x=r=2$ and $a=1$, and the $(4, 6)$-simple graph has no $(2, 3)$-factorization, and $\frac{d+s}{r+a} = \frac{4+2}{2+1} = 2 = x = \frac{4}{2} = \frac{d}{r}\,$.

\begin{figure}[h]
	\centering
		\includegraphics[width=120mm]{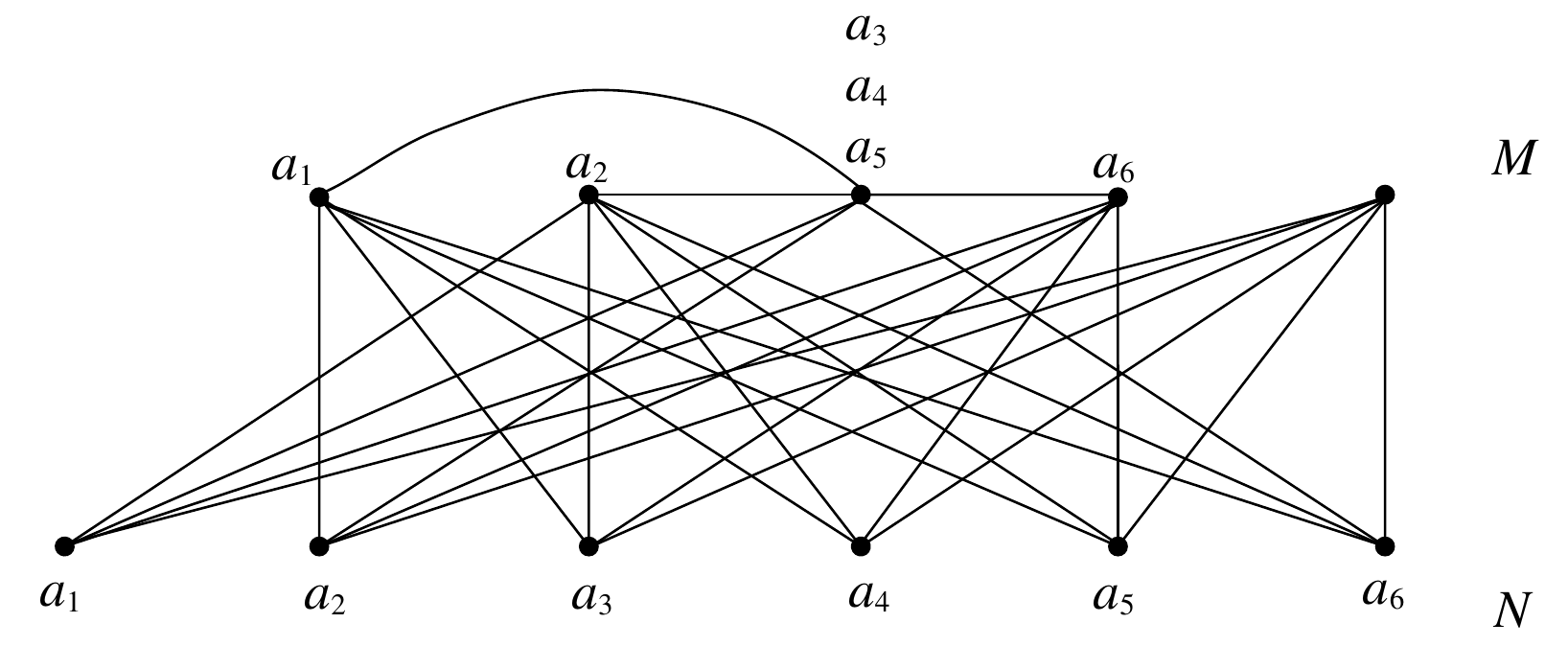}
	\caption{ A bidegreed $(4,6)$-simple graph with no $(2,3)$-factorization; a member of $BG(2,2,1,4)$.}
\end{figure}

\vspace{0.25cm}

\noindent {\bf Case 2:  Let $x$ be odd}

Let $G$ be a simple graph with vertex sets $M \cup N$ where $|M| = xr+1$ and $|N| = x(r+a)$. The vertices of $M$ have degree $x(r+a)$ and all except one vertex of $N$ will have degree $xr$ and one vertex of $N$, say $v_{x(r+a)}$, will have degree $xr+1$. Let $H$ be a simple graph with $V(H) \subset M$ and $|E(H)| = \frac{x(r+a)-1}{2}$. Label the vertices of $H$ with labels $a_1, a_2, \ldots, a_{x(r+a)-1}$ in such a way that if $v \in V(H)$ then $v$ receives $d_H(v)$ labels. Also assign the labels $a_1, \ldots, a_{x(r+a)-1}$ to the vertices of $N$, assigning one label to each vertex and leaving one vertex, say $v_{x(r+a)},$ unlabelled.

\vspace{0.25cm}
We have $H$ already placed on the vertices of $M$. To form $G$ from this, first join each vertex $v$ of $N$ to each vertex of $M$ except the vertex with the same label as $v$ (the vertex $v_{x(r+a)} \in V(N)$ is joined to all the vertices of $M$). Then, for $v \in V(M)$, $d_G(v) = x(r+a)$, and, for $v \in V(N) \backslash \{v_{x(r+a)}\}$, $d_G(v) = xr$ and $d_G(v_{x(r+a)}) = xr+1$. Then
$$\begin{array}{lll} |E(G)| & = & xr(x(r+a)) + 1 + \displaystyle \frac{x(r+a)-1}{2} \vspace{0.25cm} \\
& = & x^2 r^2 + x^2ra + \displaystyle \frac{x(r+a)+1}{2}\;.\end{array}$$
If $G$ has an $(r,r+a)$-factorization with $x$ factors, let $\{F_1, F_2, \ldots, F_x\}$ be such a set of factors.  Then each vertex of $M$ will have $r+a$ edges incident with each of $F_1, F_2, \ldots, F_x$, and, for all but one $i$, $F_i$ will have $r$ edges incident with each vertex of $N$, and the exceptional factor, say $F_x$, will have $r$ edges incident with each vertex of $V(N) \backslash \{v_{x(r+a)}\}$, and will have $r+1$ edges incident with $v_{x(r+a)}$. Therefore
$$\begin{array}{lll}\displaystyle\sum_{v\in V(G)} d_{F_i}(v)  & \ge & \left\{\begin{array}{ll}(xr+1)(r+a)+x(r+a)r & {\rm if} \;\; i \neq x\,,  \vspace{0.2cm} \\
(xr+1)(r+a)+x(r+a)r+1 & {\rm if} \;\; i = x\,, \end{array} \right. \vspace{0.25cm} \\
& = & \left\{\begin{array}{ll}2xr^2 + 2xra + (r+a) & {\rm if} \;\; i \neq x\,, \vspace{0.2cm} \\
2x^2 r^2 + 2xra + (r+a)+1 & {\rm if} \;\; i=x\,. \end{array}\right. \end{array}$$
Therefore
$$\begin{array}{lll}|E(F_i)|  & \ge & \left\{\begin{array}{ll} xr^2 + xra + \displaystyle\left\lceil\frac{r+a}{2}\right\rceil & {\rm if} \;\; i \neq x\,, \vspace{0.2cm} \\
xr^2 + xra+ \displaystyle\frac{r+a+1}{2} & {\rm if} \;\; i = x \end{array}\right. \vspace{0.25cm} \\
& = & xr^2 + xra + \displaystyle\frac{r+a+1}{2}\,.\end{array} $$
Therefore
$$\begin{array}{lll} |E(G)| & \ge & \displaystyle x^2 r^2 + x^2ra + \frac{x(r+a+1)}{2} \vspace{0.25cm} \\
& = & x^2r^2 + x^2ra + \displaystyle \frac{x(r+a)+1}{2} + \frac{(x-1)}{2} \vspace{0.25cm} \\
& > & x^2r^2 + x^2ra + \displaystyle \frac{x(r+a)+1}{2}\,, \;\; {\rm since} \; x \ge 1\,, \vspace{0.25cm} \\
& = & E(G), \end{array}$$
a contradiction.

\vspace{0.25cm}
Therefore $G$ has no $(r,r+a)$-factorization when $x\ge 3$, $x$ odd.

\end{proof}

In Figure 2 we give an example which illustrates the construction used in Theorem \ref{T30}, Case 2. Here $x=3$, $r=2$, $a=1$, $d=6$, $s=3$, so $\frac{d+s}{r+a} = \frac{d}{r} = x = 3$, and the $(6,9)$-simple graph has no $(2,3)$-factorization. 

\begin{figure}[h]
	\centering
		\includegraphics[width=140mm]{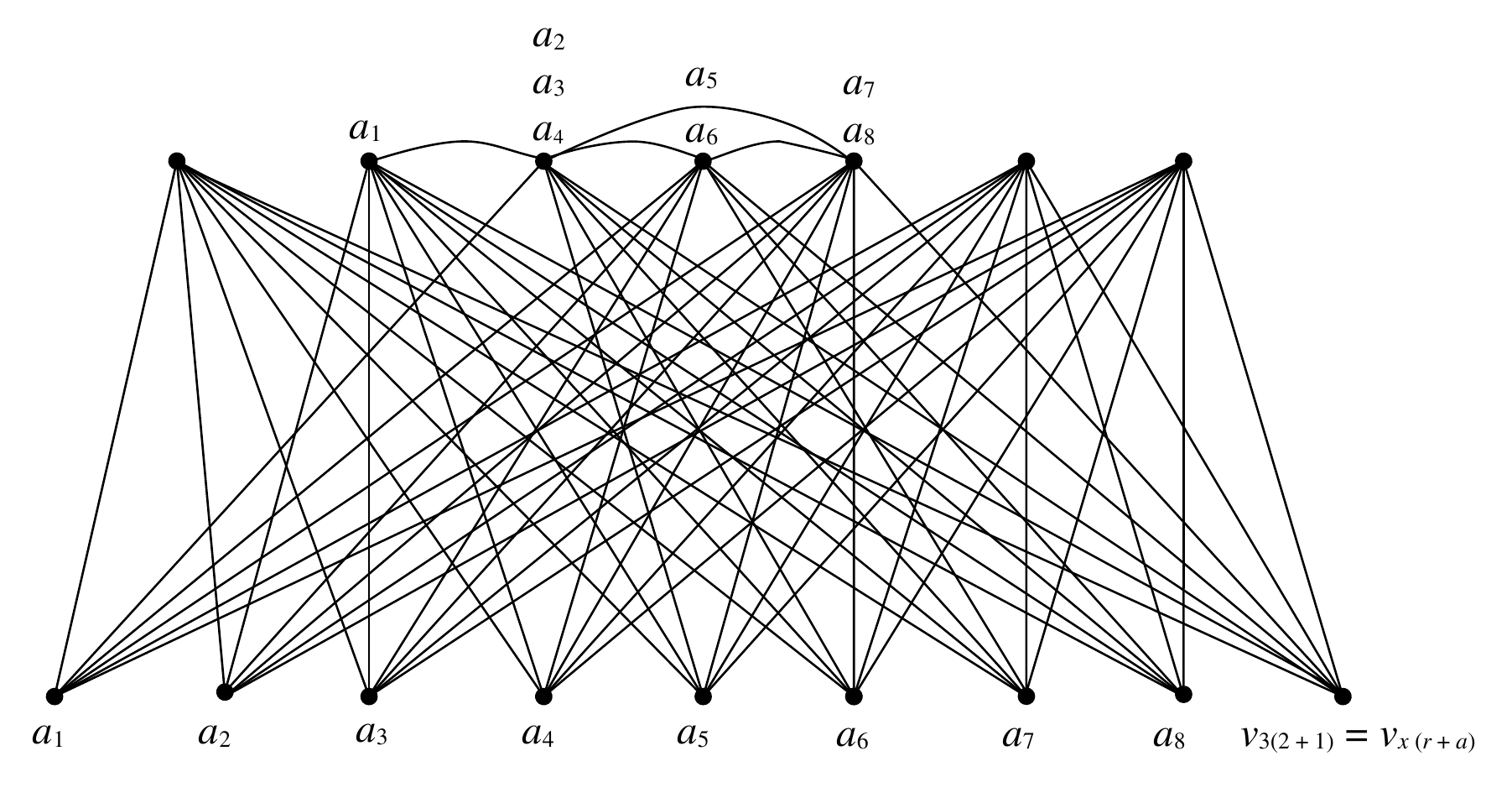}
	\caption{A $(6,9)$-simple graph with no $(2,3)$-factorization; a member of $B(2,3,1,6)$.}
\end{figure}

\vspace{0.25cm}
Similarly:

\begin{theorem} Let $r$ and $a$ be odd positive integers with $a \ge 3$. Let $d$, $s$ and $x$ be positive integers such that $\frac{d+s}{r+a} = x = \frac{d}{r}$ and $x \ge 2$. Then $B(r,s,a,d) \neq \phi$.
\label{T31}\end{theorem}

\begin{proof}

 We separate out the cases $x$ even and $x$ odd. Although these are similar, it is easier for the reader if they are treated separately.

\vspace{0.25cm}
\noindent{\bf Case 1:} Let $G$ be a bi-degreed simple graph with vertex sets $M$ and $N$ where $|M|=xr$ and $|N|=x(r+a)+1$. Let $H$ be a simple graph with $V(H) \subset N$ and $|E(H)| = \frac{xr}{2}$. Label the vertices of $H$ with labels $a_1, a_2, \ldots, a_{xr}$ in such a way that if $v \in V(H)$ then $v$ receives $d_H(v)$ labels. Also assign the labels $a_1, a_2, \ldots, a_{xr}$ to the vertices of $M$, assigning one label to each vertex.

We have $H$ placed on the vertices of $N$. Then to form $G$ from this, join each vertex $v$ of $M$ to each vertex of $N$ except the vertex with the same label as $v$. Then, for $v \in M$, $d_G(v) = x(r+a)$ and, for $v \in N$, $d_G(v) = xr$.

Notice that
$$|E(G)| \; = \; xr(x(r+a)) + \frac{xr}{2} \; = \; x^2 r^2 + x^2ra + \frac{xr}{2}\;.$$
If $G$ has an $(r,r+a)$-factorization with $x$ factors, let $\{F_1, \ldots, F_x\}$ be such a set of factors. Each $F_i$ will have $r+a$ edges incident with each vertex of $M$, so for $1 \le i \le x$, 
$$\sum_{v\in V(G)} d_{F_i}(v) \; \ge \; xr(r+a) + (x(r+a)+1)r \; = \; 2xr^2 + 2xra + r\,.$$
Therefore, for each $i$, $1 \le i \le x$, 
$$|E(F_i)| \; \ge \; xr^2 + xra + \left\lceil\frac{r}{2}\right\rceil\,.$$
Therefore
$$\begin{array}{lll} |E(G)| & \ge & x \displaystyle\left(xr^2+xra+\left\lceil\frac{r}{2}\right\rceil\right)\vspace{0.25cm} \\
& = & x^2 r^2 + x^2ra + x\displaystyle \left\lceil\frac{r}{2}\right\rceil \vspace{0.25cm} \\
& > & x^2 r^2 + x^2 ra + x\displaystyle\frac{r}{2}\,, \;\; {\rm since} \; r \; {\rm is \; odd}, \vspace{0.25cm} \\
& = & |E(G)|\,,\end{array} $$
a contradiction.

Thus $G$ has no $(r,r+a)$-factorization with $x$ factors when $x$ is even.

\vspace{0.25cm}
\noindent\textbf{An aside.}

In Figure 3 we give an example with $r=a=1$ and $x=2$.

\begin{figure}[h]
	\centering
		\includegraphics[width=110mm]{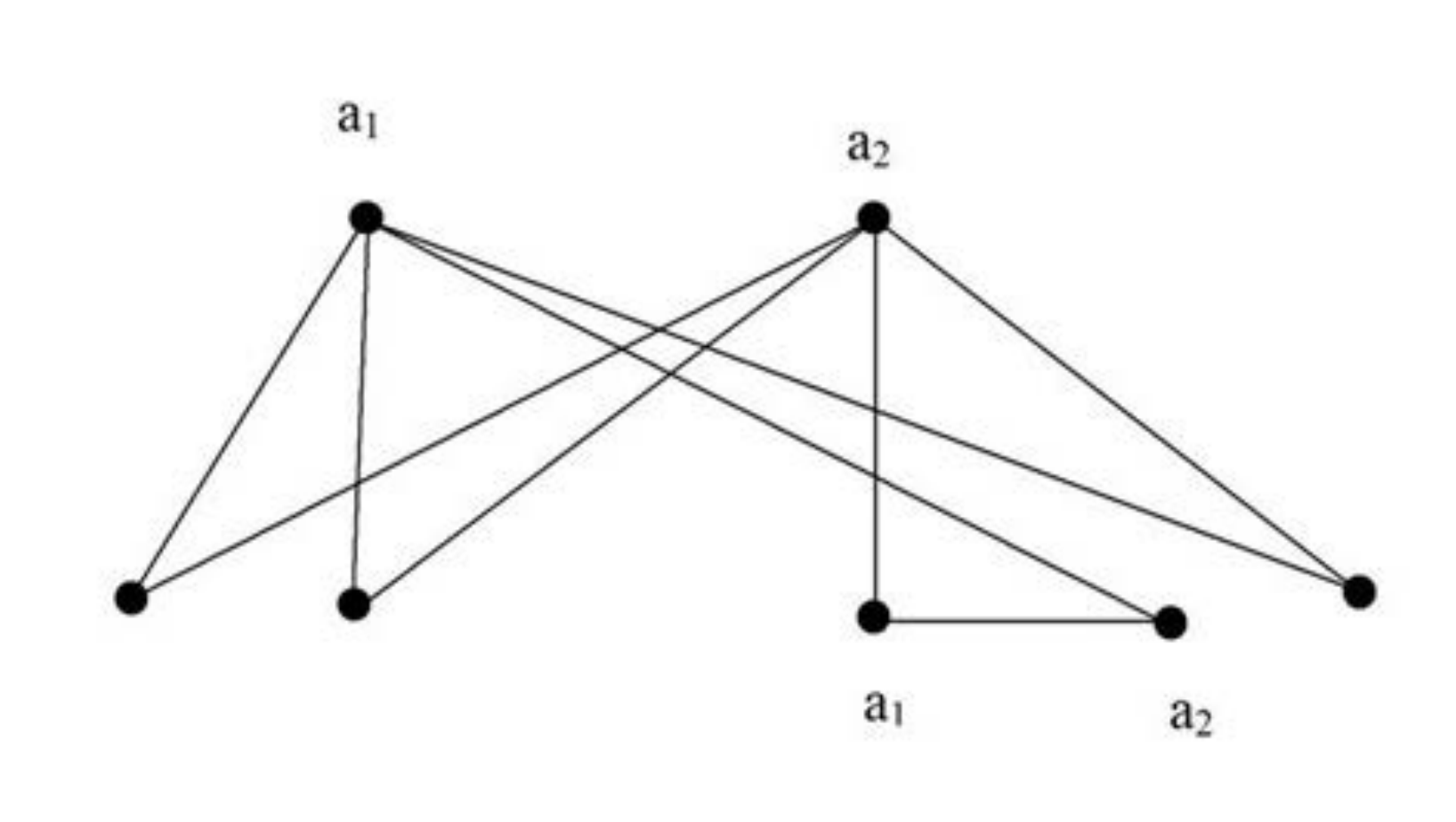}
	\caption{ A bidegreed $(2,4)$-simple graph with no $(1,2)$-factorization; an example of a graph in $B(1,2,1,2)$.}
\end{figure}

\vspace{0.25cm}
\noindent\textbf{Case 2: Let $x$ be odd}

Let $G$ be a simple graph with vertex sets $M \cup N$ where $|M| = xr$ and $|N| =  x(r+a)+1$. The vertices of $N$ will have degree $xr$ and all except one vertex of $M$ will have degree $x(r+a)$, with one vertex having degree $x(r+a)-1$. Let $H$ be a simple graph with $V(H) \subset N$ and $|E(H)| = \frac{xr+1}{2}$. Label the vertices of $H$ with labels $a_1, a_2, \ldots, a_{xr}, a_{xr+1}$ in such a way that labels $a_{xr}$ and $a_{xr+1}$ are assigned to different vertices of $H$ and, if $v \in V(H)$, then $v$ receives $d_H(v)$ labels. Also assign the labels $a_1, a_2, \ldots, a_{xr+1}$ to the vertices of $M$, with one vertex, say $v_{xr}$ receiving two labels, say $a_{xr}$ and $a_{xr+1}$, and the remaining $xr-1$ vertices receiving one label from $a_1, a_2, \ldots, a_{xr-1}$ each.

We have $H$ already placed on the vertices of $N$. To form $G$ from this, first join each vertex $v$ of $M \backslash \{v_{xr}\}$ to each vertex of $N$ except the vertex with the same label as $v$. Join $v_{xr}$ to all vertices of $N$ except the vertices with labels $a_{xr}$ and $a_{xr+1}$. Then, for $v \in N$, $d_G(v)=xr$, for $v\in M \backslash \{v_{xr}\}$, $d_{G(v)} = x(r+a)$ and $d_{G(v_x r)} = x(r+a)-1$. Then
$$\begin{array}{lll}|E(G)| & = & xr(x(r+a)) - 1 + \displaystyle \frac{xr+1}{2} \vspace{0.25cm} \\
& = & x^2 r^2 + x^2 ra + \displaystyle \frac{xr}{2} - \frac{1}{2}\;. \end{array}$$
If $G$ has an $(r,r+a)$-factorization with $x$ factors, let $\{F_1, F_2, \ldots, F_x\}$ be such a set of factors. Then, for all except one $i$, $F_i$ will have $r+a$ edges incident with each vertex of $M$, but for one $i$, say $i=x$, $F_i$ will have $r+a-1$ edges incident with $v_{xr}$, but will have $r+a$ edges incident with each other vertex of $M$. Therefore
$$\begin{array}{lll}\displaystyle\sum_{v\in V(G)} d_{F_i}(v) & = & \left\{\begin{array}{ll}xr(r+a)+(x(r+a)+1)r\,, & {\rm if} \;\;i\neq x\,, \vspace{0.2cm} \\
xr(r+a)-1+(x(r+a)+1)r\,, & {\rm if} \;\; i=x\,, \end{array}\right.  \vspace{0.25cm} \\
& = & \left\{\begin{array}{ll} 2xr^2 + 2xra + r\,, & {\rm if} \;\; i \neq x\,, \vspace{0.2cm} \\
2xr^2 + 2xra + r-1\,, & {\rm if} \;\; i=x\,.\end{array}\right. \end{array}$$
Therefore
$$|E(G)| \; \; \ge \;\; \left\{\begin{array}{ll} xr^2 + xra + \displaystyle \left\lceil\frac{r}{2}\right\rceil\,, & {\rm if} \;\; i \neq x\,, \vspace{0.25cm} \\
xr^2 + xra + \displaystyle \frac{r-1}{2}\,, & {\rm if} \;\; i=x\,.\end{array}\right.$$
Therefore
$$\begin{array}{lll}|E(G)| & \ge & (x-1)\left(xr^2 + xra + \displaystyle\left\lceil\frac{r}{2}\right\rceil\right) + xr^2 + xra + \displaystyle\left\lceil\frac{r-1}{2}\right\rceil \vspace{0.25cm} \\
& = & x^2 r^2 + x^2 ra + x \displaystyle \left\lceil\frac{r}{2}\right\rceil - 1 \vspace{0.25cm} \\
& > & x^2 r^2 + x^2 ra + \displaystyle \frac{xr}{2} - \frac{1}{2} \quad {\rm since} \;\; x \ge 3\,, \vspace{0.25cm} \\
& = & |E(G)|\,, \end{array}$$
a contradiction (noting that $x\left\lceil\frac{r}{2}\right\rceil - 1 > \frac{xr}{2} - \frac{1}{2}$ when $x \ge 3$).

Therefore $G$ has no $(r,r+a)$-factorization when $x \ge 3$, $x$ odd.

\end{proof}

\newpage

In Figure 4 we illustrate the construction used in Case 2 in Theorem \ref{T31}.

\begin{figure}[h]
	\centering
		\includegraphics[width=110mm]{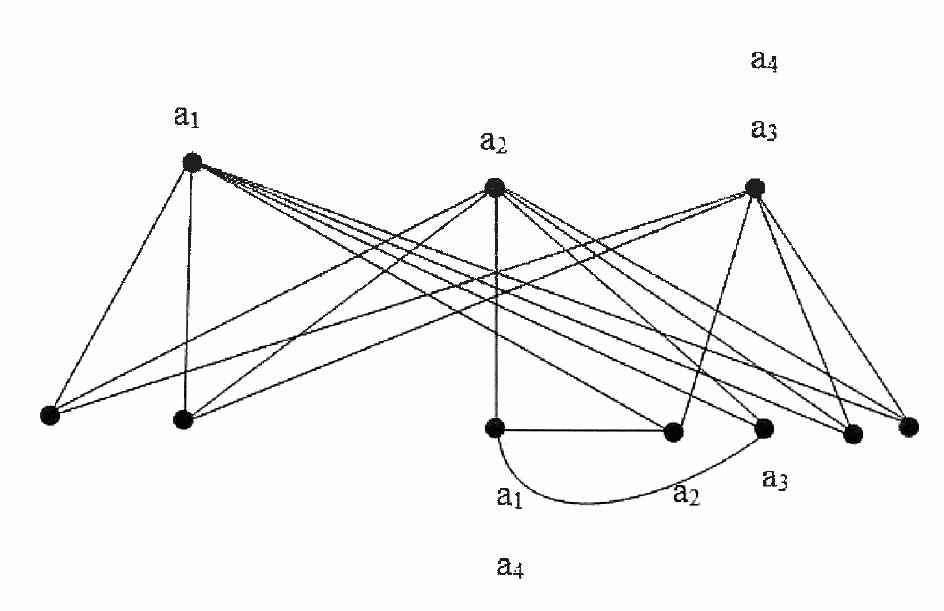}
	\caption{ A bidegreed $(3,6)$-simple graph with no $(1,2)$-factorization; an example of a graph in $B(1,3,1,3)$.}
\end{figure}

\vspace{0.5cm}
\noindent\textbf{Some further problems}

\vspace{0.25cm}
\noindent 1. C.J.H. McDiarmid has pointed out to the authors that from Theorem 4.1 of his interesting paper \cite{McDiarmid83} on unimodular matrices follows this fact about $(r,r+a)$-factorizations of bipartite multigraphs. Given a graph $G$ and non-negative integers $r$ and $a$, let $F_{\{r,a\}}(G)$ be the set of integers $x$ such that $G$ has an $(r,r+a)$-factorization with $x$ factors. Then if $G$ is a bipartite multigraph, $F_{\{r,a\}}(G)$ is an interval of integers.

He also remarked that, following arguments of the authors (which may be found in Rajkumar's thesis), it follows that if $r$ and $a$ are both even, and $G$ is any pseudograph, then $F_{\{r,a\}}(G)$ is again an interval of integers. [Recall that a pseudograph is a multigraph in which multi-loops are permitted, with a loop contributing $2$ to the degree of the vertex it is on.].

The question remains if this is also true when one or both of $a$ and $r$ is odd.

\vspace{0.25cm}
\noindent 2. It remains to determine the threshold number $\mu(r,s,a,t)$ for multigraphs (without loops). Theorem \ref{T7} seems to indicate that this will not be an easy task.

\bibliographystyle{plain}

\begin{thebibliography}{99}
\bibitem{Egawa86} Y. Egawa, {\em Era's conjecture on $[k, k+1]$-factorizations of regular graphs}, Ars Combin., {\bf 21} (1986), 217--220. 
\bibitem{Era84} H. Era, {\em Semiregular factorizations of regular graphs}, in: F. Harary, J. Maybee (Eds), Graphs and Applications, Proceedings of the First Colorado Symposium of Graph Theory, Wiley, New York, 1984, pp. 101--116. 
\bibitem{FerencakHilton10} M. Ferencak and A.J.W. Hilton,   {\em Semiregular factorization of regular multigraphs}, Mathematika, {\bf 56} (2010), 357--362.  
\bibitem{FerencakHilton11} M. Ferencak and A.J.W. Hilton, {\em Regular multigraphs and their semiregular factorizations}, Congressus Numerantium, {\bf 209} (2011), 149--159. 
\bibitem{Gupta66} R.P. Gupta, {\em A theorem on the cover index of an {\em s}-graph}, Notices Amer. Math. Soc., {\bf 13} (1966), 714. 
\bibitem{Hilton75} A.J.W. Hilton, {\em The cover index, the chromatic index and the minimum degree of a graph}, in: Proceedings of the 5th British Combinatorial Conference, Congressus Numerantium, vol. XV, 1975, pp. 307--318. 
\bibitem{Hilton08} A.J.W. Hilton, {\em $(r, r+1)$-factorizations of $(d, d+1)$-graphs}, Discrete Math., {\bf 308} (2008), 645--669. 
\bibitem{Hilton09} A.J.W. Hilton, {\em On the number of $(r, r+1)$-factors in an $(r, r+1)$-factorization of a simple graph}, J. Graph Theory, {\bf 60} (2009), 257--268. 
\bibitem{Hilton10} A.J.W. Hilton, {\em Degree bounded factorizations of bipartite multigraphs and of pseudographs}, Discrete Math., {\bf 310} (2010), 288--303. 
\bibitem{HiltonRaj1} A.J.W. Hilton and A. Rajkumar, {\em The pseudograph $(r, s, a, t)$-threshold number}, Discrete Applied Mathematics, {\bf 209} (2016), 155--163. 
\bibitem{HiltonRaj2} A.J.W. Hilton and A. Rajkumar,  {\em The simple graph threshold number $\sigma(r, s, a, t)$ when $r$ is odd and $a \ge 2$ is even}, Congressus Numerantium, \textbf{223} (2015), 33--44. 
\bibitem{HiltonDeWerra94} A.J.W. Hilton and D. de Werra, {\em A sufficient condition for equitable edge-colourings of simple graphs}, Discrete Math., {\bf 128} (1994), 179--201. 
\bibitem{HiltonWojc05} A.J.W. Hilton and J. Wojciechowski,  {\em Semiregular factorization of simple graphs}, AKCE International Jour. of Graphs and Combinatorics, {\bf 2} (2005), 57--62. 
\bibitem{KanoSaito} M. Kano and A. Saito, {\em [a,b]-factors of graphs}, Discrete Math., {\bf 47} (1983), 113--116. 
\bibitem{McDiarmid72} C.J.H. McDiarmid, {\em The solution of a time-tabling problem}, J. Inst. Math. Appl., {\bf 9} (1972), 23--34. 
\bibitem{McDiarmid83} C.J.H. McDiarmid, {\em Integral decomposition in polyhedra}, Mathematical Programming, {\bf 25} (1983), 183--198. 
\bibitem{Vizing64} V.G. Vizing, {\em On an estimate of the chromatic class of a $p$-graph}, Diskret. Analiz., {\bf 3} (1964), 25--30 (in Russian). 
\bibitem{DeWerra71} D. de Werra, {\em Equitable colorations of graphs}, Rev. fran\c caise d'Inform. Recherche Operationelle, {\bf R-3} (1971), 3--8. 
\bibitem{ZhangLiu11} Xia Zhang and Guizhen Liu, {\em Equitable edge-colorings of simple graphs}, J. Graph Theory, {\bf 66} (2011), 175--197. 

\end{thebibliography}

\vspace{0.5cm}
\noindent A.J.W. Hilton, Department of Mathematics and Statistics, University of Reading, Whiteknights, Reading RG6 6AX, UK  \\
email: a.j.w.hilton@reading.ac.uk \\
\noindent Also:  School of Mathematical Sciences, Queen Mary University of London, Mile End Road, London E1 4NS, UK \\
email: a.hilton@qmul.ac.uk

\vspace{0.25cm}
\noindent A. Rajkumar, School of Mathematical Sciences, Queen Mary University of London, Mile End Road, London E1 4NS, UK \\
email: a.rajkumar@qmul.ac.uk

\end{document}